\documentclass[12pt]{article}
\usepackage[english]{babel}
\usepackage[utf8]{inputenc}
\usepackage[T1]{fontenc}
\usepackage{tikz}
\usepackage{graphicx}
	\graphicspath{{images/}}
\usepackage{url}
\usepackage{enumitem}
\usepackage[all]{xy}
\usepackage[colorlinks,citecolor=blue,urlcolor=blue]{hyperref}
\usepackage{makeidx}
\makeindex
\usepackage{amsmath,amsfonts,amsthm,amssymb}
\usepackage{typearea}
\usepackage{graphicx}
\usepackage{geometry}
\newcommand{\stsets}[1]{\mathbb{#1}}
\newcommand{\R}{\stsets{R}}

\newcommand{\N}{\stsets{N}}
\newcommand{\Z}{\stsets{Z}}

\theoremstyle{plain}

\newtheorem{theorem}{Theorem}[section]
\newtheorem{Prop}{Proposition}[section]
\newtheorem{lemma}[theorem]{Lemma}
\theoremstyle{remark}
\newtheorem{definition}[theorem]{Definition}
\newtheorem{assumption}[theorem]{Assumption}

\renewcommand{\P}{\mathbf{P}}
\DeclareMathOperator{\E}{{\bf E}}
\DeclareMathOperator{\one}{{ 1\hspace*{-0.55ex}I}}
\DeclareMathOperator{\card}{card}
\DeclareMathOperator{\Supp}{Supp}
\DeclareMathOperator*{\argmax}{arg\,max}
\DeclareMathOperator{\TLLN}{TLLN}
\newcommand*\diff{\mathop{}\!\mathrm{d}}

\newcommand{\eps}{\varepsilon}

\renewcommand{\P}{\mathbf{P}}

\renewcommand{\epsilon}{\varepsilon}

\renewcommand{\phi}{\varphi}

\newlength{\querylen}
\usepackage{fancybox}

\def\acknowledgementsname{Acknowledgments}
\newenvironment{acks}[1][\acknowledgementsname]{\section*{#1}}{\par}
 
\if@balayout
    \renewenvironment{acks}[1][\acknowledgementsname]%
        {%
            \vskip0.5\baselineskip
            \small
            {\noindent\normalfont\sffamily\bfseries\acknowledgementsname}\par
            \begingroup\parindent 0pt\parskip 0.5\baselineskip
        }%
        {\endgroup}
\fi
\def\fundingname{Funding}

\title{Propagation of chaos and Poisson Hypothesis for replica mean-field models of intensity-based neural networks}
\author{Michel Davydov \thanks{INRIA, Paris, France and D\'epartement d'informatique de l'ENS, ENS, CNRS, PSL University, Paris, France, michel.davydov@inria.fr} }
\date{}
\begin{document}
\maketitle
\begin{abstract}
Neural computations arising from myriads of interactions between spiking neurons can be modeled as network dynamics with punctuate interactions. However, most relevant dynamics do not allow for computational tractability. To circumvent this difficulty, the Poisson Hypothesis regime replaces interaction times between neurons by Poisson processes. We prove that the Poisson Hypothesis holds at the limit of an infinite number of replicas in the replica-mean-field model, which consists of randomly interacting copies of the network of interest. The proof is obtained through a novel application of the Chen-Stein method to the case of a random sum of Bernoulli random variables and a fixed point approach to prove a law of large numbers for exchangeable random variables.
\end{abstract}
\section{Introduction}

Many phenomena in a variety of fields can be modeled as punctuate interactions between agents. Whether it be opinion dynamics \cite{Amblard_2004}, epidemics propagation \cite{Pastor_Satorras_2015}, flow control on the internet \cite{baccelli:2002} or neural computations \cite{Shriki_2013}, an agent-based approach is a versatile way to describe the phenomenon of interest through the behavior of each agent.

In such an approach, each agent is seen as a node in a network in which edges represent the possibility of interactions, and point processes associated to each node register the times at which these interactions happen. These point processes idealize the stochasticity inherent in the phenomena of interest. The state of the system can then be given by a set of stochastic differential equations, each describing the state of an agent. In intensity-based models, used extensively for example in computational neuroscience \cite{truccolo_2005}, \cite{Pillow2008}, this state is given by the stochastic intensity of the point process associated with the agent. Neuronal stochastic intensities model the instantaneous firing rate of a neuron as a function of the spiking inputs received from other neurons.

However, the versatility of this "microscopic" approach comes at a cost, namely, that of computational tractability. Indeed, except for the simplest network architectures, such as systems of 1 or 2 agents, an analytical expression characterizing the law of the GL model in the stationary regime is not in general available. To go beyond numerical simulations, it then becomes imperative to resort to some simplifying assumption.

As the complexity of the dynamics resides in the dependencies between agents due to interactions, it is natural to choose a simplified model in which the agents are considered independent. One such classical approach is called the mean-field regime. Introduced originally by McKean \cite{McKean_66} and developed, among others, by Dobrushin\cite{Dobrushin1979} and Sznitman\cite{Szn:89}, it consists in approximating the interactions received by any one particle by an empirical mean of the interactions over the whole network. The mean-field regime arises at the limit with infinitely many agents, as the empirical mean typically converges to an expectation and gives rise to a nonlinear ordinary differential equation describing the behavior of an agent at a macroscopic level and allowing for tractability. This convergence, when it takes place, is linked to the concept of propagation of chaos, mainly in reference to the asymptotic independence between agents that arises at the limit. 

In classical mean-field models such as \cite{Fournier_Locherbach_2016} or \cite{LeBoudec07}, the network considered must be assumed fully connected, the effect of interactions on the state of a given agent must be small, typically inversely proportional to the number of agents, in order to prevent explosion in finite time in the system. These assumptions represent significant constraints on the architectures and sizes of the networks and thus on the types of phenomena for which a mean-field approximation is relevant.

To circumvent these limitations, different approaches have been explored in recent years. To apply mean-field approximation to small-sized networks (with less than 100 agents, for example), the refined-mean-field approach \cite{Gast_2018}, \cite{gast_allmeier_22} adds a corrective term to the macroscopic ODE. Mean-field models have also been studied in other scalings, for example diffusive, where the effect of interactions on a particle is inversely proportional to the square root of the number of agents in the system. So-called conditional propagation of chaos properties have been proven in that setting \cite{Erny_21}. To incorporate heterogeneity, the properties of graphons (large dense graphs) have been used to derive new limit equations \cite{ZAN_2022}. In this setting, the limit object is an infinite system of ODEs. However, this approach is only valid for dense networks; when the average degree of a node is of order lesser than the amount of nodes in the network, as is the case for example in the human brain, graphon theory does not apply. Another approach to incorporate heterogeneity circumvents mean-field models altogether, relying instead on conditional independence properties and local weak limits to obtain local convergence \cite{Ramanan_2020}.

In this work, we are interested in another tractable regime for agent-based models: the Poisson Hypothesis. First formulated by Kleinrock for large queueing systems \cite{Kleinrock_1}, it states that the flow of arrivals to a given node can be approximated by a Poisson flow with rate equal to the average rate of the original flow of arrivals. In agent-based models, the flow of arrivals corresponds to the effect of interactions on a given node. Under the Poisson Hypothesis, the behavior of each agent is still described by a stochastic differential equation, but the agents are considered independent and interaction times are replaced by Poisson clocks, which allows for tractability. This regime has been studied for queueing models by Rybko, Shlosman and others \cite{Vladimirov_2018} and by Baccelli and Taillefumier for intensity-based models from computational neuroscience \cite{Baccelli_2019}.

Hereafter, we focus on the continuous-time Galves-Löcherbach model introduced in \cite{Galves_2013} and studied under the Poisson Hypothesis in \cite{Baccelli_2019}. We introduce a physical system, called the replica-mean-field, derived from the initial model that converges under a certain scaling to the dynamics under the Poisson Hypothesis. The replica-mean-field was first introduced by Dobrushin to study queueing models\cite{VveDobKar96}, and adapted to a network setting in \cite{Baccelli_2019}. However, in their work, the convergence of the replica-mean-field dynamics to dynamics under the Poisson Hypothesis is only assumed, and not proven. The crux of this article is proving that a propagation of chaos-type convergence does take place for the replica-mean-field model derived from the Galves-Löcherbach model.

In the recent work \cite{Davydov2022}, we have introduced a class of discrete-time processes on a discrete space, called fragmentation-interaction-aggregation processes (FIAPs), that include among others discrete versions of the Galves-Löcherbach model, and we have proven the propagation of chaos property for a replica-mean-field model associated with a FIAP for any finite time. Our aim is to generalize these results to a model in continuous time and with a continuous state space.

\subsection*{Structure of the paper}
After this general introduction, we formally define all the models that we will be considering, namely, the Galves-Löcherbach model, its replica-mean-field version and its dynamics under the Poisson Hypothesis. We then state the main result of the paper, namely propagation of chaos in the replica-mean-field model on compacts of time, which we then prove in Section \ref{sec_PH_proof}. Finally, in Section \ref{sec_Rel_comp_RMF_LGL}, we generalize the main result to weak convergence on the half-line through a tightness argument.

\subsection*{The continuous-time Galves-L\"ocherbach (GL) model}

Let us formally present the Galves-Löcherbach model mentioned above: we consider a finite collection of $K$ neurons whose spiking activities are given by the  realization of a system of simple point processes without any common points $\mathbf{N}=\{N_i\}_{1 \leq i \leq K}$ on $\mathbb{R}$ defined on some measurable space $(\Omega,\mathcal{F})$. For each neuron $1 \leq i \leq K$, we denote by $(T_{i,n})_{n \in \mathbb{Z}}$ the sequence of successive spiking times with the convention that, almost surely, $T_{i,0} \leq 0 < T_{i,1}$ and  $T_{i,n}<T_{i,n+1}$ for all $n$. 

To model the interactions due to spikes within the system, we consider that the spiking rate of neuron $i$ depends on the times at which neuron $i$ and the other neurons $j \neq i$ have spiked in the past. Formally, we introduce the network spiking history  $(\mathcal{F}_t)_{t\in \mathbb{R}}$ as an increasing collection of $\sigma$-fields such that
\begin{equation*}
\mathcal{F}_t^\mathbf{N}=\{\sigma(N_1(B_1),...,N_K(B_K))|B_i\in \mathbf{B}(\mathbb{R}), B_i \subset (-\infty,t]\}\subset \mathcal{F}_t,
\end{equation*}
where $\mathcal{F}_t^\mathbf{N}$ is the internal history of the spiking process $\mathbf{N}$.

Recall that the $\mathcal{F}_t$-stochastic intensity $\{\lambda_i(t)\}_{t \in \mathbb{R}}$ of the associated point process $N_i$ is the $\mathcal{F}_t$-predictable random process satisfying for all $s<t \in \mathbb{R}:$
\begin{equation}
\label{eq_stoch_int}
\E\left[N_i(s,t]|\mathcal{F}_s\right]=\E\left[\int_s^t \lambda_i(u)\diff u\big|\mathcal{F}_s\right],
\end{equation}
where $\mathcal{F}_t$ is the network history. We will hereafter refer to \eqref{eq_stoch_int} as the stochastic intensity property.

We consider the $\mathcal{F}_t$-stochastic intensities $\lambda_1,...,\lambda_K$ associated with the point processes $N_1,...,N_K$. 

In the Galves-L\"ocherbach model, the evolution in time of these intensities is given by the following system of stochastic differential equations:
\begin{equation}
\label{eq_LGL_SDE}
\lambda_i(t)=\lambda_i(0)+\frac{1}{\tau_i}\int_0^t\left(b_i-\lambda_i(s)\right)\diff s+\sum_{j \neq i}\mu_{j\rightarrow i}\int_0^t N_j(\diff s)+\int_0^t\left(r_i-\lambda_i(s)\right) N_i(\diff s),
\end{equation}
where $\tau_i, b_i, r_i >0$ and $\mu_{j\rightarrow i} \geq 0$ are given constants and $\lambda_i(0)$ is assumed to be greater than $r_i$ and $b_i$. 

Let us make each term more explicit. The first integral term shows that without any spikes, the intensity exponentially decays to its base rate $b_i$ with a relaxation time $\tau_i$. The second integral term represents the aggregation of all the spikes received from the other neurons in the system: a spike received from neuron $j$ causes a jump of $\mu_{j\rightarrow i}$ units in the intensity of neuron $i$. Finally, the third integral term is the reset that occurs when neuron $i$ spikes: $\lambda_i$ is then reset to $r_i$, which is a value such that $0 < r_i \leq b_i$. Taking $r_i \leq b_i$ models a refractory period that occurs after a spike during which the neuron enters a transient phase.

It has been shown in \cite{Baccelli_2019} that \eqref{eq_LGL_SDE} defines a piecewise deterministic Markov process. When there is no exponential decay, that is, when $\tau_i=\infty$ for all $i$, the process becomes a pure jump Markov process that is Harris-ergodic and thus has a unique invariant measure. 

It has also been shown that the moment generating function (MGF) at stationarity\\ $u\rightarrow L(u)=\E[\exp(\sum_{i=1}^K u_i\lambda_i)]$ satisfies the following differential equation:
\begin{equation}
\label{eq_MGF_LGL}
\left(\sum_i\frac{u_i b_i}{\tau_i}\right)L-\sum_i\left(1+\frac{u_i}{\tau_i}\right)\partial_{u_i}L+\sum_i e^{(u_ir_i+\sum_{j \neq i}u_j\mu_{j \rightarrow i})}\partial_{u_i}L_{|u_i=0}=0.
\end{equation} 

This equation is not solvable except for some very special cases ($K \leq 2,$ for example). It is thus necessary to use approximating schemes or truncating moments, both of which neglect correlations due to the finite size of the network. Here, we introduce a different physical system that allows to obtain closed forms for equations similar to \eqref{eq_MGF_LGL}.

\subsection*{Replica-mean-fields of GL models}
In replica-mean-field models, we consider $M$ replicas of the initial set of $K$ neurons. When neuron $i$ in replica $m$ spikes, for each neuron $j$ that would receive something from the spike, a replica $n$ is uniformly and independently chosen among the other $M-1$ replicas, and neuron $i$ sends $\mu_{i \rightarrow j}$ to it.

Formally, for $1 \leq m \leq M, 1 \leq i,j \leq K,$ let $\{V^M_{(m,i)\rightarrow j}(t)\}_{t\in \mathbb{R}}$ be $\mathcal{F}^N$-predicatble stochastic processes such that, for each spiking time $T$, i.e., each point of $N^M_{m,i}$, the random variables $\{V^M_{(m,i)\rightarrow j}(T)\}_j$ are mutually independent, independent from the past, and uniformly distributed on $\{1,...,M\}\setminus\{m\}.$
Here, $V^M_{(m,i)\rightarrow j}(T)$ gives the index of the replica to which the spike of neuron $i$ in replica $m$ at time $T$ is sent to neuron $j$.

The stochastic intensities associated with the point processes will then solve the following system of stochastic differential equations:

\begin{equation}
\label{eq_RMF_LGL_SDE}
\begin{split}
\lambda^M_{m,i}(t)&=\lambda^M_{m,i}(0)+\frac{1}{\tau_i}\int_0^t\left(b_i-\lambda^M_{m,i}(s)\right)\diff s \\
&+\sum_{j \neq i}\mu_{j\rightarrow i}\sum_{n \neq m} \int_0^t \one_{\{V^M_{(n,j)\rightarrow i}(s)=m\}} N^M_{n,j}(\diff s)+\int_0^t \left(r_i-\lambda^M_{m,i}(s)\right) N^M_{m,i}(\diff s).
\end{split}
\end{equation}

These equations, which we will hereafter refer to as the RMF dynamics, characterize the dynamics of the $M$-replica system. As before, for all $M$, these dynamics form a pure jump Harris-ergodic Markov process with a unique invariant measure since we are under Assumption \ref{Ass_1}.

The infinitesimal generator of the $M$-replica dynamics is given by 
\begin{equation*}
\begin{split}
\mathcal{A}[f](\mathbf{\lambda})&= \sum_{i=1}^K\sum_{m=1}^M\left(\frac{b_i-\lambda_{m,i}}{\tau_i}\partial_{\lambda_{m,i}} f(\mathbf{\lambda})\right) \\
&+\sum_{i=1}^K\sum_{m=1}^M \frac{1}{|\mathcal{V}_{m,i}|}\sum_{v \in \mathcal{V}^M_{m,i}} \left( f\left(\mathbf{\lambda}+\mathbf{\mu}_{m,i,v}(\mathbf{\lambda})\right)-f(\mathbf{\lambda})\right)\lambda_{m,i},
\end{split}
\end{equation*}
where the update due to the spiking of neuron $(m,i)$ is defined by
\begin{equation*}
\left[\mathbf{\mu}_{m,i,v}(\mathbf{\lambda})\right]_{n,j}= 
\begin{cases}
\mu_{j \rightarrow i} & \text{ if } j \neq i, n=v_j \\
r_i - \lambda_{m,i} & \text{ if } j=i,n=m \\
0 & \text{ otherwise }.
\end{cases}
\end{equation*}
As before, we can establish an equation for the MGF \\
$u\rightarrow L(u)=\E[e^{(\sum_{m=1}^M\sum_{i=1}^K u_{m,i}\lambda_{m,i})}]$ in the stationary regime:
\begin{equation*}
\begin{split}
&\left(\sum_m\sum_i\frac{u_{m,i} b_i}{\tau_i}\right)L-\sum_m\sum_i\left(1+\frac{u_{m,i}}{\tau_i}\right)\partial_{u_{m,i}}L\\
&+\sum_m\sum_i \frac{1}{(M-1)^K}\sum_{v \in \mathcal{V}^M_{m,i}} e^{(u_{m,i}r_i+\sum_{j \neq i}u_{v_j,j}\mu_{j \rightarrow i})}\partial_{u_{m,i}}L_{|u_i=0}=0,
\end{split}
\end{equation*} 
where 
\begin{equation*}
\mathcal{V}^M_{m,i}=\{v\in\{1,...,M\}^K|v_i=m \textrm{ and } v_j \neq m, j\neq i\}.
\end{equation*}

Once again, this equation is not easily solvable. However, a closed form has been obtained in \cite{Baccelli_2019} by setting the Poisson Hypothesis, that is, by considering that at the limit when $M \rightarrow \infty$, the replicas become asymptotically independent and the arrivals process to a given replica becomes asymptotically Poisson.

This hypothesis is often conjectured or numerically validated and not proven, as was the case in \cite{Baccelli_2019}. The aim of this work is to give a proof of the Poisson Hypothesis in the RMF limit in continuous time, by analogy of the work done in discrete time in \cite{Davydov2022}.

In the rest of the paper, we make the following assumptions:
\begin{assumption}
\label{Ass_1}
For all $i \in  \{1,\ldots,K\}$, $\tau_i=\infty$ (no exponential decay).
\end{assumption}

\begin{assumption}
\label{Ass_2}
There exists $\xi_0>0$ such that for all $1\leq m \leq M, 1 \leq i \leq K$ and all $0<\xi\leq \xi_0$, $\E[e^{\xi\lambda_{m,i}(0)}]<\infty$.
\end{assumption}
Assumption \ref{Ass_1} is introduced mainly to simplify notation and reasonings. While we do not rigorously prove it, we strongly believe that all results remain valid without it.
Assumption \ref{Ass_2} restricts the initial conditions of the system to allow for propagation of chaos to take place. We shall see that this assumption allows us to have bounds for the moments of the state process and later to apply Chernoff's inequality at a crucial juncture.

\subsection*{The limit dynamics}
In this section, we aim to define the limit dynamics of the RMF GL model when the number of replicas goes to infinity. As previously mentioned, intuitively, the arrivals from each neuron should become a Poisson process, whereas the reset term should remain unchanged.

As such, we introduce the following system of SDEs which is the natural candidate for the limit dynamics and to which we will hereafter refer to as the limit process. We will denote with tildes everything pertaining to it.
We consider point processes $\Tilde{N}_1,\ldots,\Tilde{N}_K$ on $\R^+$  with respective $(\mathcal{F}_t)$-stochastic intensities $\Tilde{\lambda}_1,\ldots,\Tilde{\lambda}_K$, where $\mathcal{F}_t$ is the internal spiking history of the network defined as previously, verifying the following stochastic differential equations: for $t>0,$ for $1\leq i \leq K,$ 
\begin{equation}
\label{eq_limit_eq_nonexch}
\Tilde{\lambda}_i(t)=\Tilde{\lambda}_i(0)+\sum_{j \neq i}\mu_{j\rightarrow i} \tilde{A}_{j \rightarrow i}(t)+\int_0^t \left(r_i-\Tilde{\lambda}_{i}(s)\right) \tilde{N}_{i}(\diff s),
\end{equation}
where $\tilde{A}_{j \rightarrow i}$ are independent inhomogeneous Poisson point processes with intensities $a_{j}(t)=\int_0^t\E[\tilde{\lambda}_j(s)]\diff s=\E[\tilde{N}_j([0,t])]$ and $(\Tilde{\lambda}_1(0),\ldots,\Tilde{\lambda}_K(0))$ verify Assumption \ref{Ass_2}.

The existence and uniqueness of the solution to this equation comes from the general theory of \cite{Robert2016}, and is derived analogously to the existence and uniqueness of the solution to \eqref{eq_RMF_LGL_SDE}, which is done in \cite{Baccelli_2019}. Note that \eqref{eq_limit_eq_nonexch} is a nonlinear equation in the sense of McKean-Vlasov \cite{McKean_66}, as the process $\Tilde{\lambda}_i$ depends on its own law through the presence of the terms $\E[\tilde{\lambda}_j(t)]$ in the intensities of the Poisson processes.

\subsection*{The main result}
Recall the following definition of convergence in total variation:

\begin{definition}
\label{def_tv_dist}
Let $P$ and $Q$ be two probability measures on a probability space $(\Omega, \mathcal{F})$. We define the total variation distance by
\begin{equation*}
d_{TV}(P,Q)=\sup_{A\in \mathcal{F}}|P(A)-Q(A)|.
\end{equation*}
When $\Omega$ is countable, an equivalent definition is
\begin{equation*}
d_{TV}(P,Q)=\frac{1}{2}\sum_{\omega \in \Omega}|P(\omega)-Q(\omega)|.
\end{equation*}
\end{definition}
Note that certain authors use a multiplicative constant 2 when defining the total variation distance. We will also abusively say that random variables converge in total variation when their distributions do.

The following theorem is the main result of this work:

\begin{theorem}
\label{th_Poisson_CV}
There exists  $T>0$ such that, if $t \in [0,T]$ and if
\begin{equation*}
A^M_{m,i}(t)=\sum_{j \neq i}\mu_{j\rightarrow i} \sum_{n \neq m}  \int_0^t \one_{\{V^M_{(n,j)\rightarrow i}(s)=m\}} N^M_{n,j}(\diff s),
\end{equation*}
with $N^M_{m,i}$ defined in \eqref{eq_RMF_LGL_SDE},
and
\begin{equation*}
\tilde{A}_i(t)=\sum_{j \neq i}\mu_{j\rightarrow i}  \tilde{A}_{j \rightarrow i}(t),
\end{equation*}
with $\tilde{A}_{j \rightarrow i}(t)$ defined in \eqref{eq_limit_eq_nonexch}, then, 
\begin{enumerate}
\item the processes $(\tilde{A}_1,\ldots,\tilde{A}_K)$ are independent, as are the processes $(\tilde{\lambda}_1,\ldots,\tilde{\lambda}_K);$
\item for all $(m,i) \in \{1,\ldots,M\}\times\{1,\ldots,K\}$, the random variable $A^M_{m,i}(t)$ converges in total variation to $\tilde{A}_i(t)$ when $M \rightarrow \infty;$
\item for all $(m,i) \in \{1,\ldots,M\}\times\{1,\ldots,K\}$, the random variable $\lambda^M_{m,i}(t)$ defined by \eqref{eq_RMF_LGL_SDE} converges in total variation to $\Tilde{\lambda}_i(t)$ defined in \eqref{eq_limit_eq_nonexch} when $M \rightarrow \infty;$ 
\item let $\mathcal{N}$ be a finite subset of $\N^*$, for all $i\in \{1,\ldots,K\},$ the processes $(A^M_{m,i}(\cdot))_{m \in \mathcal{N}}$ and $(\lambda^M_{m,i}(\cdot))_{m \in \mathcal{N}}$ weakly converge in the Skorokhod space $D([0,T])^{\card{\mathcal{N}}}$ endowed with the product Skorokhod metric to $\card(\mathcal{N})$ independent copies of the corresponding limit processes $(\tilde{A}_i(\cdot))$ and $(\tilde{\lambda}_i(\cdot))$ when $M \rightarrow \infty$.
\end{enumerate}

\end{theorem}

Here are a few remarks on this result. 

First, note that for each $i,j \in \{1,\ldots,K\},$ the variable $\mu_{j\rightarrow i}\tilde{A}_{j \rightarrow i}(t)$ is a scaled Poisson random variable and is thus a special case of a compound Poisson random variable. As such, unless all $\mu_{j\rightarrow i}$ are equal, $\tilde{A}_i(t)$ does not follow any standard Poisson or compound Poisson distribution.

Note also that we do not aim to prove $\mathcal{L}^1$ convergence, which we believe does not hold in this model. This marks a significant difference with classical mean field models for Hawkes processes, see \cite{Fournier_Locherbach_2016}. From a computational point of view, this is due to the fact that the averaging factor $\frac{1}{M-1}$ only appears in expectation in the interaction term $A^M_{m,i}(t)$.

Finally, note that statement 4 of the theorem is a consequence of statements 2 and 3, as weak convergence of point processes is implied by weak convergence of their void probabilities (see \cite{Kallenberg_2017}, Theorem 2.2), which comes directly from statements 2 and 3. More precisely, the convergence in total variation of $\lambda_{m,i}$ and $A_{m,i}$ implies the weak convergence of the void probabilities of $N_{m,i}$ and of the point process of arrival times, which in turn implies weak convergence of the point processes.

\subsection*{Methodology for the proof}
In contrast with classical mean-field models presented in the beginning of the paper, in replica mean-fields, the mean-field simplification comes from the random routing operations between replicas. The input point process in the $M$-replica model consists in a superposition of $M$ rare point processes, which informally explains why Poisson (or compound Poisson) processes arise at the limit. This point of view leads us to fix an instant $t \in \R^+$ and to consider the random variable of inputs up to time $t$ as a random sum of Bernoulli random variables with means $\frac{1}{M-1}$. The Chen-Stein method is a natural candidate to obtain explicit bounds in the total variation metric between this random sum and a Poisson random variable. We generalize it to account for the fact that the amount of Bernoulli random variables in the sum is random. As far as the author is aware, this application is novel.
The bound obtained through the Chen-Stein method does not easily converge to 0 when $M$ goes to infinity. Namely, we obtain a term similar to the $L^1$ norm of an empirical mean of centered random variables that are not independent. In order to circumvent the difficult direct analysis of such a term, we uncouple the inputs and outputs of the dynamics by considering the replica-mean-field dynamics \eqref{eq_RMF_LGL_SDE} as the fixed point of some function on the space of probability distributions on the space of càdlàg functions endowed with a metric rendering it complete. This procedure is often used in the study of stochastic differential equations to prove the existence and uniqueness of solutions to these equations, see, for example, \cite{Szn:89} or \cite{BremMass96}. Here, we use it to prove that a certain property, namely the convergence of an empirical mean, holds at the fixed point by proving that the property is preserved by the function and that the iterates of the function converge to its fixed point.

\section{Proof of the Poisson Hypothesis for the RMF GL model}
\label{sec_PH_proof}
The aim of this section is to prove Theorem \ref{th_Poisson_CV}. We organize the proof as follows. First, we recall some well-known facts about Poisson embedding representations for real-valued point processes and derive asymptotic independence from them. Then, we state and prove some properties of the RMF and limit dynamics that will be useful in the following steps of the proof. Afterwards, we present the Chen-Stein method, which we use to derive conditions for the validity of the Poisson approximation that we aim to prove. Finally, we interpret the RMF dynamics as the fixed point of some function on the space of probability measures on the space of càdlàg functions and we show that this function has properties that allow the aforementioned conditions to hold at the fixed point, thus proving the Poisson approximation result.

\subsection*{Poisson embedding representation and independence of the limit processes}
First, recall the following result from \cite{Bremaud20} about Poisson embeddings:
\begin{lemma}
\label{lem_poisson_embedd}
Let $N$ be a point process on $\mathbb{R}$. Let $(\mathcal{F}_t)$ be an internal history of $N$. Suppose $N$ admits a $(\mathcal{F}_t)$-stochastic intensity $\{\mu(t)\}_{t \in \mathbb{R}}.$ Then there exists a Poisson point process $\hat{N}$ with intensity 1 on $\mathbb{R}^2$ such that, for all $C \in \mathcal{B}(\mathbb{R}),$
\begin{equation*}
    N(C)=\int_{C \times \mathbb{R}}\one_{[0, \mu(s)]}(u)\hat{N}(\diff s\times \diff u).
\end{equation*}
\end{lemma}
This result states that any point process admitting a stochastic intensity can be embedded in a Poisson point process with intensity 1 on the positive half-plane by considering the points of the Poisson process which lie below the curve given by the stochastic intensity of the process. We now proceed to apply this in our model, constructing all the state processes coupled through their Poisson embeddings.

For $m \geq 1, M \geq 1, 1 \leq i \leq K,$ let $\hat{N}_{m,i}$ be i.i.d. Poisson point processes on $\R^+ \times \R^+$ with intensity 1.

Let $\Omega=(\mathbb{R}^+ \times ((\mathbb{R}^+)^2)^{\mathbb{N}^*})^{\mathbb{N}^*}$ be a probability space endowed with the probability measure $(\mu_0 \otimes P)^{\otimes \mathbb{N}^*},$ where $\mu_0$ is the law of the initial conditions and $P$ is the law of a Poisson process with intensity 1 on $(\mathbb{R}^+)^2$.
We construct on $\Omega$ the following processes:
\begin{itemize}
    \item The processes $(N^M_{m,i}(t)), m \geq 1, M \geq 1, 1 \leq i \leq K,$ with stochastic intensities $(\lambda^M_{m,i}(t))$ verifying
    \begin{equation}
    \label{eq_rmf_coupling}
    \begin{split}
    \lambda^M_{m,i}(t)&=\sum_{n \neq m} \sum_{j \neq i} \int_0^t\int_0^{+\infty} \mu_{j\rightarrow i}\one_{\{V^M_{(n,j)\rightarrow i}(s)=m\}} \one_{[0,\lambda^M_{n,j}(s)]}(u) \hat{N}_{n,j}(\diff s \times \diff u)\\
    &+\int_0^t \int_0^{+\infty}\left(r_i-\lambda^M_{m,i}(s)\right) \one_{[0,\lambda^M_{m,i}(s)]}(u) \hat{N}_{m,i}(\diff s \times \diff u)+\lambda^M_{m,i}(0),
    \end{split}
    \end{equation}
    with $\lambda^M_{m,i}(0)=Z_{i}$ for all $m \in \N^*$ and where, for all $M$, $(V^M_{(n,j)\rightarrow i}(t))_j$ are càdlàg stochastic processes such that for each point $T$ of $\hat{N}_{n,j}$, the random variables $(V^M_{(n,j)\rightarrow i}(T))_j$ are independent of the past, mutually independent and uniformly distributed on $\{1,...,M\}\setminus \{n\}$, considered as marks of the Poisson point process $\hat{N}_{n,j}$. Namely, to each point of the Poisson embedding, we attach a mark that is an element of $(\N^{K})^{N^*}$, where the $M$th term of the sequence corresponds to $(V^M_{(n,j)\rightarrow i}(T))_j$.
    \item The processes $(\tilde{N}_{i}(t)), 1 \leq i \leq K,$ with stochastic intensities $(\tilde{\lambda}_{i}(t))$ verifying
    \begin{equation}
    \label{eq_limit_coupling}
    \begin{split}
    \Tilde{\lambda}_{i}(t)&=\Tilde{\lambda}_{i}(0)+\sum_{j \neq i} \int_0^t\int_0^{+\infty}\mu_{j\rightarrow i}\one_{[0,\E[\tilde{\lambda}_{j}(s)]]}(u)\hat{N}_{j,i}(\diff s \times \diff u)\\
    &+\int_0^t \int_0^{+\infty}\left(r_i-\Tilde{\lambda}_{i}(s)\right) \one_{[0,\tilde{\lambda}_{i}(s)]}(u)\hat{N}_{i,i}(\diff s \times \diff u),
    \end{split}
    \end{equation}
    with $\Tilde{\lambda}_{i}(0)=Z_{i}$.
\end{itemize}

In other words, we construct the $M$-replica dynamics \eqref{eq_RMF_LGL_SDE} and the limit processes \eqref{eq_limit_eq_nonexch} with the same initial conditions and Poisson embeddings $(\hat{N}_{m,i})$.\\
We require that the law of the initial conditions $(Z_{i})$ verifies Assumption \ref{Ass_2}, in other words, we require it to have uniform polynomial bounds of its moments.

This representation allows us to derive the following, which is statement 1 of Theorem \ref{th_Poisson_CV}.
\begin{lemma}
\label{lem_asympt_indep}
The processes $(\tilde{A}_{j \rightarrow i})_{1 \leq i,j \leq K}$ are independent, as are the processes $(\tilde{\lambda}_1,\ldots,\tilde{\lambda}_K)$.
\end{lemma}
\begin{proof}
For all $t \in [0,T]$, we can write using the construction above
\begin{equation*}
\tilde{A}_{j \rightarrow i}(t)=\int_0^t\int_0^{+\infty}\one_{[0,\E[\tilde{\lambda}_{j}(s)]]}(u)\hat{N}_{j,i}(\diff s \times \diff u).
\end{equation*}

Therefore, all the randomness in $\tilde{A}_i$ is contained in the Poisson embeddings $(\hat{N}_{k,i})_{1 \leq k \leq K}$. Thus, for $i \neq j$, $\tilde{A}_i$ and $\tilde{A}_j$ are independent. The independence of the processes $(\tilde{\lambda}_1,\ldots,\tilde{\lambda}_K)$ follows in the same manner.
\end{proof}
\subsection*{Properties of the RMF and limit processes}
\label{subsec_RMF_properties}
In this section, we prove several properties of the RMF and limit dynamics that will be used throughout the proof.

In what follows, we will often omit the $M$ superscript in the notations $N^M_{m,i}$, $A^M_{m,i}$ and $\lambda^M_{m,i}$ to increase readability.

Note that the arrival process $A_{m,i}(t)$ can be represented as a random sum of Bernoulli random variables with parameters $\frac{1}{M-1}$. Indeed, for $n \neq m$ and $j\neq i$, if $S \in \Supp (N_{n,j}|_{[0,T)})$, we define $B^M_{S,(n,j)\rightarrow (m,i)}$ the random variable equal to 1 if the routing between replicas at time $S$ caused by a spike in neuron $j$ in replica $n$ chose the replica $m$ for the recipient $i$ of the spike, and 0 otherwise. As such, it is clear that we can write for all $t \in [0,T], m \in \{1,\ldots,M\}$ and $i\in \{1,\ldots,K\},$ 
\begin{equation}
\label{eq_arr_rep}
A_{m,i}(t)=\sum_{n \neq m}\sum_{j \neq i}\mu_{j\rightarrow i}\sum_{k \leq N_{n,j}([0,t])}B^M_{k,(n,j)\rightarrow (m,i)}.
\end{equation}
Note that when $m,n,i$ and $j$ are fixed, the random variables $(B^M_{k,(n,j)\rightarrow (m,i)})_{k \leq N_{n,j}([0,T])}$ are i.i.d. Also note that when $n,j,i$ and $k$ are fixed, the joint distribution of $(B^M_{k,(n,j)\rightarrow (m,i)})_m$ with $m \in \{1,\ldots,M\}$ is that of Bernoulli random variables with parameter $\frac{1}{M-1}$ such that exactly one of them is equal to 1, all the others being equal to 0.
Combining these two observations allows us to prove a lemma that highlights the core of the replica-mean-field approach:

\begin{lemma}
\label{lem_cond_indep_bern}
Fix $(m,i) \in \{1,\ldots,M\}\times\{1,\ldots,K\}$. Keeping notation from \eqref{eq_arr_rep}, let \\
$N=(N_{n,j}([0,t]))_{n \neq m, j \neq i} \in  \mathbb{N}^{(K-1)(M-1)}.$

Conditionally on the event $\{N=q\},$ for $q=(q_{n,j})_{n \neq m, j \neq i} \in \mathbb{N}^{(K-1)(M-1)},$ the random variables $(B^M_{k,(n,j)\rightarrow (m,i)})_{n \neq m, j \neq i, k \in \{1,\ldots,q_{n,j}\}}$ are independent Bernoulli random variables with parameter $\frac{1}{M-1}.$
\end{lemma}
\begin{proof}
The structure of the proof goes as follows: since $N$ is entirely determined by the Poisson embeddings $(\hat{N}_{n,j})_{j\neq i}$ and the arrivals to the nodes $(n,j)$ from all the nodes $h\neq j$ across replicas, it is sufficient to show that these arrivals and the routing variables \\
$(B^M_{k,(n,j)\rightarrow (m,i)})_{k\leq\hat{N}_{n,j}([0,t]\times\R^+)}$ are independent. Intuitively, this holds because arrivals are aggregated across all replicas, which will erase the eventual dependencies due to the routing variables to nodes $i$ choosing one replica rather than another.\\
In order to transcribe this intuition rigorously, we first show that the total number of departures from nodes $i$ up to time $t$, that is, $\sum_{l=1}^M N_{l,i}([0,t])$, and the routing variables \\
$(B^M_{k,(n,j)\rightarrow (m,i)})_{k\leq\hat{N}_{n,j}([0,t]\times\R^+)}$ are independent. Indeed, using the representation given by Lemma \ref{lem_poisson_embedd}, due to the structure of the Poisson embeddings $(\hat{N}_{l,i})_{l\in \{1,\ldots,M\}}$, there is a point of $\sum_{l=1}^M N_{l,i}$ in some interval $I$ iff there is a point of the superposition of the corresponding Poisson embeddings such that the x-coordinate is in $I$ and the y-coordinate is under the curve of the function $t\rightarrow \sum_{l=1}^M\lambda_{l,i}(t)$. In turn, the last event does not depend on $(B^M_{k,(n,j)\rightarrow (m,i)})_{k\leq\hat{N}_{n,j}([0,t]\times\R^+)}$, as the symmetry inherent to the replica structure ensures that all arrivals increment $t\rightarrow \sum_{l=1}^M\lambda_{l,i}(t)$ by the same amount, which concludes the proof of this preliminary remark.

For all $(n,j)$ such that $n \neq m$ and $j \neq i$, let \begin{equation*}A_{i \rightarrow (n,j)}(t)=\sum_{l\neq n}\sum_{k\leq N_{l,i}([0,t])}B^M_{k,(l,i)\rightarrow (n,j)}.
\end{equation*}
Note that $A_{i \rightarrow (n,j)}(t)$ represents the arrivals to node $j$ in replica $n$ from all nodes $i$ across replicas. As such, it is clear that we can write
\begin{equation*}
A_{i \rightarrow (n,j)}(t)=\sum_{k\leq \sum_{l \neq n} N_{l,i}([0,t])}B^M_{k,(i)\rightarrow (n,j)},
\end{equation*}
where $(B^M_{k,(i)\rightarrow (n,j)})$ are independent Bernoulli random variables with parameter $\frac{1}{M-1}$ such that they and $(B^M_{k,(n,j)\rightarrow (m,i)})$ are independent. Then by the previous observation, $A_{i \rightarrow (n,j)}(t)$ and $(B^M_{k,(n,j)\rightarrow (m,i)})$ are independent. Therefore, $N$, which is entirely determined by the Poisson embeddings $(\hat{N}_{n,j})$ and the arrivals $(A_{h \rightarrow (n,j)}(t))_{h\neq j}$, and
$(B^M_{k,(n,j)\rightarrow (m,i)})_{k\leq\hat{N}_{n,j}([0,T]\times\R^+)}$, are independent. Thus, conditioning on $N$ does not break independence between the variables $(B^M_{k,(n,j)\rightarrow (m,i)}).$
\end{proof}

We will now give bounds on the moments of both the M-replica and limit processes, using the bounds on the moments of the initial conditions.
\begin{lemma}
\label{lem_moment_bound}
Suppose the initial conditions verify Assumption \ref{Ass_2}.
Then, for all $p\geq 1$, for all $(m,i) \in \{1,\ldots,M\}\times\{1,\ldots,K\}$, for all $t \in [0,T]$, there exists $Q_p(T) \in \mathbb{R}_p[X]$ a polynomial of degree exactly $p$ such that
\begin{equation}
\label{eq_gronwall_p}
\E[\lambda_{m,i}^{p}(t)]\leq Q_p(\E[\lambda_{m,i}(0)]).
\end{equation}
\end{lemma}
\begin{proof}
We first prove the result for $p=1$ in the special case where exchangeability also holds between neurons and where there are no resets. Namely, we temporarily consider here that for all $i,j \in \{1,\ldots,K\}, \mu_{j\rightarrow i}=1$.
Let $t\in [0,T]$.
Thus, we have
\begin{equation*}
\E[\lambda_{m,i}(t)]=\E[\lambda_{m,i}(0)]+\sum_{n \neq m} \sum_{j \neq i} \E\left[\int_0^t \one_{\{V_{(n,j)\rightarrow i}(s)=m\}} N_{n,j}(\diff s)\right].\\
\end{equation*}

Denoting by $((T_{n,j})_r)_{r\in \Z}$ the points of $N_{n,j},$ we have
\begin{equation*}
\begin{split}
\E[\lambda_{m,i}(t)]&=\E[\lambda_{m,i}(0)]\\
&+\sum_{n \neq m} \sum_{j \neq i} \E\left[\sum_{r\in \Z}\E[ \one_{\{V_{(n,j)\rightarrow i}((T_{n,j})_r)=m\}} \one_{\{(T_{n,j})_r\in [0,t)\}}|\mathcal{F}^{N_{n,j}}_{(-\infty,(T_{n,j})_r}]\right].
\end{split}
\end{equation*}

Using the predictability w.r.t. the network history of the processes $V^M_{(n,j)\rightarrow i}$, we have
\begin{equation*}
\E[\lambda_{m,i}(t)]=\E[\lambda_{m,i}(0)]+\sum_{n \neq m} \sum_{j \neq i} \E\left[\sum_{r\in \Z}\one_{\{V_{(n,j)\rightarrow i}((T_{n,j})_r)=m\}} \one_{\{(T_{n,j})_r\in [0,t)\}}\right].
\end{equation*}
Using the property of stochastic intensity, we can rewrite this as
\begin{equation*}
\E[\lambda_{m,i}(t)]=\E[\lambda_{m,i}(0)]+(K-1) \E\left[\int_0^t \lambda_{m,i}(s)\diff s\right].
\end{equation*}
Therefore, 
\begin{equation*}
\E[\lambda_{m,i}(t)]= \E[\lambda_{m,i}(0)]+(K-1)\int_0^t \E[\lambda_{m,i}(s)]\diff s.
\end{equation*}

This gives
\begin{equation}
\label{eq_gronwall_1}
\E[\lambda_{m,i}(t)]\leq \E[\lambda_{m,i}(0)]e^{(K+r_i-1)T}:=Q_1(\E[\lambda_{m,i}(0)]).
\end{equation}
Now, let us write the differential equation for $\lambda_{m,i}^2$, still considering the dynamics without resets and with equal weights:
\begin{equation*}
\lambda_{m,i}^{2}(t)=\lambda_{m,i}^{2}(0)+\sum_{n \neq m} \sum_{j \neq i} \int_0^t \one_{\{V_{(n,j)\rightarrow i}(s)=m\}}\left(2\lambda_{n,j}(s)+1\right) N_{n,j}(\diff s).
\end{equation*}
Therefore, we have
\begin{equation*}
\E[\lambda_{m,i}^{2}(t)]=\E[\lambda_{m,i}^{2}(0)]+(K-1) \int_0^t \left(2\E[\lambda_{m,i}^2(s)]+\E[\lambda_{m,i}(s)]\right)\diff s 
\end{equation*}
which gives using \eqref{eq_gronwall_1} the bound
\begin{equation*}
\E[\lambda_{m,i}^{2}(t)]\leq \E[\lambda_{m,i}^{2}(0)]+(K-1)Q_1((\E[\lambda_{m,i}(0)])T+2(K-1)\int_0^t \E[\sup_{u \in [0,s]}\lambda_{m,i}^2(u)] \diff s.
\end{equation*}

By applying Gronwall's lemma and using the assumption on the initial conditions, we get
\begin{equation*}
\E[\lambda_{m,i}^{2}(t)]\leq \left((\E[\lambda_{m,i}(0)]^2+(K-1)Q_1((\E[\lambda_{m,i}(0)])T\right)e^{2(K-1)T}:=Q_2((\E[\lambda_{m,i}(0)]).
\end{equation*}
This reasoning can be extended by induction to all $p\geq 3$, which proves the result for the case of exchangeable interactions. 

Now, to get the result in the general case where all $\mu_{j\rightarrow i}$ are not necessarily all equal to 1, note that by monotonicity the dynamics that we consider (in both the exchangeable and nonexchangeable cases) are stochastically dominated by the same dynamics without the reset terms. Finally, note that nonexchangeable dynamics without the reset terms are stochastically dominated by exchangeable dynamics without the reset terms. This shows that the moment bounds still hold in the general case.
\end{proof}
Finally, note that the exact same reasoning can be applied to obtain an equivalent result for the limit process, which we will only state:
\begin{lemma}
\label{lem_limit_moment_bound}
For all $p\geq 1$, for all $i \in \{1,\ldots,K\}$, for all $t \in [0,T]$, there exists a polynomial $\tilde{Q}_p \in \mathbb{R}_p[X]$ a polynomial of degree exactly $p$ such that
\begin{equation}
\label{eq_gronwall_tilde_p_cFIAP}
\E[\tilde{\lambda}_{i}^{p}(t)]\leq \tilde{Q}_p[\E[\tilde{\lambda}_{i}(0)]].
\end{equation}
\end{lemma}

Lemma \ref{lem_moment_bound} allows us to prove the following result, which states that Assumption \ref{Ass_2} can be propagated to any time $t$ less than some fixed $T.$

\begin{lemma}
\label{lem_exp_moment_bound}
There exists $T>0$ and $\xi_0>0$ (which is the same as in Assumption \ref{Ass_2}) such that for $ t\in [0,T]$ and $\xi\leq \xi_0,$
\begin{equation}
\label{eq_exp_mom_bound}
\E[e^{\xi\lambda_{m,i}(t)}]<\infty \textrm{ and } \E[e^{\xi\tilde{\lambda}_{i}(t)}]<\infty.
\end{equation}
\end{lemma}
\begin{proof}
To prove this result, we once again consider exchangeable dynamics without resets, using the same observation as previously, namely that nonexchangeable dynamics with resets are stochastically dominated by exchangeable dynamics without resets, to generalize the result. Let us thus assume $\mu_{j\rightarrow i}=\mu$ for all $i,j.$
Let $\xi_0$ as in Assumption \ref{Ass_2}. Let $t\in [0,T].$
Let us write out the equation verified by $e^{\xi\lambda_{m,i}(t)}$:
\begin{equation*}
e^{\xi\lambda_{m,i}(t)}=e^{\xi\lambda_{m,i}(0)}+\sum_{j\neq i}\sum_{n\neq m}\int_0^t\one_{\{V^M_{(n,j)\rightarrow i}(s)=m\}}e^{\xi\lambda_{m,i}(s)}(e^{\xi\mu}-1)N_{n,j}(\diff s).
\end{equation*}

Taking the expectation and using the stochastic intensity property, we get
\begin{equation*}
\E[e^{\xi\lambda_{m,i}(t)}]=\E[e^{\xi\lambda_{m,i}(0)}]+\frac{1}{M-1}\sum_{j\neq i}\sum_{n\neq m}\int_0^t\E[e^{\xi\lambda_{m,i}(s)}(e^{\xi\mu}-1)\lambda_{n,j}(s)]\diff s.
\end{equation*}
Using exchangeability between replicas, this boils down to
\begin{equation*}
\E[e^{\xi\lambda_{m,i}(t)}]=\E[e^{\xi\lambda_{m,i}(0)}]+\sum_{j\neq i}\int_0^t\E[e^{\xi\lambda_{m,i}(s)}(e^{\xi\mu}-1)\lambda_{m,j}(s)]\diff s.
\end{equation*}
Since we are looking at dynamics without resets, $\lambda_{m,i}(s)$ and $\lambda_{m,j}(s)$ are positively correlated. Therefore, we have
\begin{equation*}
\E[e^{\xi\lambda_{m,i}(t)}]\leq \E[e^{\xi\lambda_{m,i}(0)}]+(e^{\xi\mu}-1)\sum_{j\neq i}\int_0^t\E[e^{\xi\lambda_{m,i}(s)}]\E[\lambda_{m,j}(s)]\diff s.
\end{equation*}
By Lemma \ref{lem_moment_bound} and Assumption \ref{Ass_2}, we have the existence of a constant $B>0$ such that
\begin{equation*}
\E[e^{\xi\lambda_{m,i}(t)}]\leq \E[e^{\xi\lambda_{m,i}(0)}]+(e^{\xi\mu}-1)(K-1)B\int_0^t\E[e^{\xi\lambda_{m,i}(s)}]\diff s.
\end{equation*}
The desired result follows from Grönwall's lemma.
\end{proof}

The final lemma we will state and prove in this section concerns the means of the RMF and limit processes. Namely, we show that the replica mean-field construction preserves the mean as $M$ varies and that the mean of the limit process coincides with the mean of the RMF process.

\begin{lemma}
\label{lem_moment_equality}
For all $M \geq 2,$ for all $(m,i) \in \{1,\ldots,M\}\times\{1,\ldots,K\},$ if the initial conditions are such that for all $(m,i) \in \{1,\ldots,M\}\times\{1,\ldots,K\},$ 
\begin{equation*}
\lambda_{m,i}(0)=\tilde{\lambda}_i(0),
\end{equation*}
then there exists $T \in \R$, such that for all $t \in [0,T],$ we have
\begin{equation*}
\E[A_{m,i}(t)]=\E[\tilde{A}_i(t)]
\end{equation*}  
and 
\begin{equation*}
\E[\lambda_{m,i}(t)]=\E[\tilde{\lambda}_i(t)].
\end{equation*}
\end{lemma}
\begin{proof}
Using both the property of stochastic intensity and exchangeability between replicas, we have as previously,
\begin{equation*}
\E[A_{m,i}(t)]=\sum_{j \neq i} \mu_{j\rightarrow i}\E[N_{m,j}([0,t))]=\sum_{j \neq i} \mu_{j\rightarrow i}\int_0^t\E[\lambda_{m,j}(s)]\diff s.
\end{equation*}
Similarly, we have
\begin{equation*}
\E[\tilde{A}_i(t)]=\sum_{j \neq i} \mu_{j\rightarrow i}\int_0^t\E[\tilde{\lambda}_{j}(s)]\diff s.
\end{equation*}
Therefore, we see that
\begin{equation}
\label{eq_diff_means_arrivals}
\left|\E[A_{m,i}(t)-\tilde{A}_i(t)]\right|=\left|\sum_{j \neq i} \mu_{j\rightarrow i}\int_0^t\E[\lambda_{m,j}(s)-\tilde{\lambda}_{j}(s)]\diff s\right|.
\end{equation}
Thus, it is sufficient to show that $\E[\lambda_{m,j}(s)-\tilde{\lambda}_{j}(s)]=0$.

Now, let $t\in [0,T]$.
Let $i(t)=\argmax_{j \in \{1,\ldots,K\}}\left|\E[\lambda_{m,j}(t)-\tilde{\lambda}_{j}(t)]\right|$ for $t \in [0,T]$. 
Now, using \eqref{eq_RMF_LGL_SDE} and \eqref{eq_limit_eq_nonexch}, we have
\begin{equation*}
\begin{split}
\left|\E[\lambda_{m,i(t)}(t)-\tilde{\lambda}_{i(t)}(t)]\right| &\leq \bigg|\E[\lambda_{m,i(t)}(0)-\tilde{\lambda}_{i(t)}(0)]+\E[A_{m,i(t)}(t)-\tilde{A}_{i(t)}(t)]\\
&+\E[\int_0^t \left(r_i-\lambda_{m,i(t)}(s)\right) \lambda_{m,i(t)}(s)-\left(r_i-\Tilde{\lambda}_{i(t)}(s)\right) \tilde{\lambda}_{i(t)}(s)\diff s]\bigg|\\
&\leq 0+\sum_{j\neq i(t)}\mu_{i(t),j}\int_0^t \left| \E[\lambda_{m,j}(s)-\tilde{\lambda}_{j}(s)]\right| \diff s\\
&+r_i\int_0^t \left| \E[\lambda_{m,i(t)}(s)-\tilde{\lambda}_{i(t)}(s)]\right| \diff s\\
&+\left|\int_0^t \E[\lambda_{m,i(t)}^2(s)-\tilde{\lambda}^2_{i(t)}(s)]\diff s\right|,
\end{split}
\end{equation*}
using the coupling on the initial conditions and \eqref{eq_diff_means_arrivals}. 

Let $C>0$. Let $A_C(t,[0,T]):=\{(\omega, t) \in \Omega \times [0,T]| \max(\lambda_{m,i(t)}(t, \omega),\tilde{\lambda}_{i(t)}(t, \omega))<C\}$. \\
Then, we write
\begin{equation*}
\begin{split}
\left|\int_0^t\E[\lambda_{m,i(t)}^2(s)-\tilde{\lambda}^2_{i(t)}(s)]\diff s\right|&=\bigg|\int_0^t\int_{A_C(s,[0,T])}(\lambda_{m,i(t)}^2(s)-\tilde{\lambda}^2_{i(t)}(s))\diff s P(\diff \omega)\\
&+\int_0^t\int_{A^c_C(s,[0,T])}(\lambda_{m,i(t)}^2(s)-\tilde{\lambda}^2_{i(t)}(s))\diff s P(\diff \omega)\bigg|.
\end{split}
\end{equation*}
Since we have $\E[\lambda_{m,i(t)}^2(s)-\tilde{\lambda}^2_{i(t)}(s)]=\E[(\lambda_{m,i(t)}(s)-\tilde{\lambda}_{i(t)}(s))(\lambda_{m,i(t)}(s)+\tilde{\lambda}_{i(t)}(s))]$, by definition of $A_C([0,T])$ the first term can be bounded by
\begin{equation}
\label{eq_first_bound}
\left|\int_0^t\int_{A_C(s,[0,T])}(\lambda_{m,i(t)}^2(s)-\tilde{\lambda}^2_{i(t)}(s))\diff s P(\diff \omega)\right| \leq 2C\left| \int_0^t\E[\lambda_{m,i(t)}(s)-\tilde{\lambda}_{i(t)}(s)]\diff s\right|.
\end{equation}

For the second term, we write
\begin{equation*}
\begin{split}
&\left|\int_0^t\int_{A^c_C(s,[0,T])}(\lambda_{m,i(t)}^2(s)-\tilde{\lambda}^2_{i(t)}(s))\diff s P(\diff \omega)\right|= \\
&\left|\int_0^t\E[(\lambda_{m,i(t)}^2(s)-\tilde{\lambda}^2_{i(t)}(s))\one_{A^c_C(s,[0,T])}]\diff s\right|.
\end{split}
\end{equation*}
Now, using the Cauchy-Schwarz inequality,
\begin{equation*}
\left|\E[(\lambda_{m,i(t)}^2(s)-\tilde{\lambda}^2_{i(t)}(s))\one_{A^c_C(s,[0,T])}]\right| \leq \sqrt{\E[(\lambda_{m,i(t)}^2(s)-\tilde{\lambda}^2_{i(t)}(s))^2]\E[\one_{A^c_C(s,[0,T])}]}.
\end{equation*}
Using Lemmas \ref{lem_moment_bound} and \ref{lem_limit_moment_bound} that give bounds on the moments of the considered processes, we have the existence of a positive constant $D_1(T)$ such that
\begin{equation*}
\sqrt{\E[(\lambda_{m,i(t)}^2(s)-\tilde{\lambda}^2_{i(t)}(s))^2]} \leq D_1(T).
\end{equation*}
For the last term, we have by definition of $A^c_C(s,[0,T]),$
\begin{equation*}
\E[\one_{A^c_C(s,[0,T])}]=\P(\max(\lambda_{m,i(t)}(s), \tilde{\lambda}_{i(t)}(s))>C).
\end{equation*}
By Lemma \ref{lem_exp_moment_bound}, there exists $T>0$ such that $\E[e^{6T\max(\lambda_{m,i(t)}(s), \tilde{\lambda}_{i(t)}(s))}]$ is finite.
Applying the Chernoff inequality, we have 
\begin{equation*}
\P(e^{6T\max(\lambda_{m,i(t)}(s), \tilde{\lambda}_{i(t)}(s))}>e^{6CT})\leq \E[e^{6T\max(\lambda_{m,i(t)}(s), \tilde{\lambda}_{i(t)}(s))}]e^{-6CT}.
\end{equation*}
Using Lemma \ref{lem_moment_bound}, this shows that there exists a constant $D_2(T)>0$ such that
\begin{equation*}
\sqrt{\E[\one_{A^c_C(s,[0,T])}]}\leq D_2(T)e^{-3CT}.
\end{equation*}
Combining the previous bounds, we finally obtain
\begin{equation}
\label{eq_second bound}
\left| \int_0^t\int_{A^c_C(s,[0,T])}\left(\lambda_{m,i(t)}^2(s)-\tilde{\lambda}^2_{i(t)}(s)\right)\diff s P(\diff \omega) \right|\leq D_1(T)D_2(T)Te^{-3CT}.
\end{equation}
Combining \eqref{eq_first_bound} and \eqref{eq_second bound}, we have
\begin{equation*}
\begin{split}
\left|\int_0^t\E[(\lambda_{m,i(t)}^2(s)-\tilde{\lambda}^2_{i(t)}(s))]\diff s\right|&\leq 2C\left|\int_0^t\E[\lambda_{m,i(t)}(s)-\tilde{\lambda}_{i(t)}(s)]\diff s\right|\\
&+D_1(T)D_2(T)Te^{-3CT}.
\end{split}
\end{equation*}

Therefore, we have
\begin{equation*}
\begin{split}
\left|\E[\lambda_{m,i(t)}(t)-\tilde{\lambda}_{i(t)}(t)]\right| &\leq D_1(T)D_2(T)Te^{-3CT}\\
&+\left(r_i+2C\right)\int_0^t \left| \E[\lambda_{m,i(t)}(s)-\tilde{\lambda}_{i(t)}(s)]\right| \diff s\\
&+\sum_{j\neq i(t)}\mu_{i(t),j}\int_0^t \left| \E[\lambda_{m,j}(s)-\tilde{\lambda}_{j}(s)]\right| \diff s.
\end{split}
\end{equation*}

By definition of $i(t)$, we then have

\begin{equation*}
\begin{split}
\left|\E[\lambda_{m,i(t)}(t)-\tilde{\lambda}_{i(t)}(t)]\right| &\leq D_1(T)D_2(T)e^{-3CT}T\\
&+\left(r_i+2C\right)\int_0^t \left| \E[\lambda_{m,i(s)}(s)-\tilde{\lambda}_{i(s)}(s)]\right| \diff s\\
&+\sum_{j\neq i(t)}\mu_{i(t),j}\int_0^t \left| \E[\lambda_{m,i(s)}(s)-\tilde{\lambda}_{i(s)}(s)]\right| \diff s.\\
&\leq D_1(T)D_2(T)e^{-3CT}T\\
&+\left(r_i+2C+\sum_{j\neq i(t)}\mu_{i(t),j}\right)\int_0^t \left| \E[\lambda_{m,i(s)}(s)-\tilde{\lambda}_{i(s)}(s)]\right| \diff s.
\end{split}
\end{equation*}

By Gronwall's lemma, we therefore get
\begin{equation*}
\left|\E[\lambda_{m,i(t)}(t)-\tilde{\lambda}_{i(t)}(t)]\right| \leq D_1(T)D_2(T)e^{-3CT}Te^{(r_i+2C+\sum_{j\neq i(t)}\mu_{i(t),j})T}.
\end{equation*}
Finally, note that for any $\epsilon >0$, we can now choose $C>0$ such that \begin{equation*}
D_1(T)D_2(T)Te^{-3CT}e^{(r_i+2C+\sum_{j\neq i(t)}\mu_{i(t),j})T}\leq \eps.
\end{equation*}
By the choice of $i(t)$, the result follows for any $1 \leq i \leq K.$
\end{proof}

\subsection*{Poisson approximation bound using the Chen-Stein method}
\label{subsec_Poisson_bound}
The goal of this section is to use the Chen-Stein method, which we will briefly recall, to obtain a bound in total variation distance between the arrivals term \eqref{eq_arr_rep} and the limit sum of Poisson random variables.
Recall that \eqref{eq_arr_rep} states that for all $t \in [0,T], m \in \{1,\ldots,M\}, i\in \{1,\ldots,K\},$ 
\begin{equation*}
A_{m,i}(t)=\sum_{n \neq m}\sum_{j \neq i}\mu_{j\rightarrow i}\sum_{k \leq N_{n,j}([0,t])}B^M_{k,(n,j)\rightarrow (m,i)}.
\end{equation*}

Recall that if $Z$ is a random variable taking values in $\N$ with $\E[Z]<\infty$, $Z$ is a Poisson random variable iff the distribution of $Z+1$ is equal to the distribution of the size-biased version of $Z$, in other words, iff for all bounded functions $f$ on $\N$,
\begin{equation}
\label{eq_chen_stein}
\E[Z]\E[f(Z+1)]=\E[f(Z)Z].
\end{equation}
The key principle of the Chen-Stein method is to say that if \eqref{eq_chen_stein} holds approximately for some r.v. $Z$ for any bounded function $f$ on $\N$, then $Z$ approximately has a Poisson distribution.
In the case of a sum of Bernoulli random variables that are not necessarily independent, we have the following result:
\begin{lemma}
\label{lem_chenstein_standard}
Let $l\in \N.$
Consider $W=\sum_{i=1}^l Y_i,$ where $Y_i$ are Bernoulli random variables with respective means $p_i$, without any independence assumptions. Let $Z$ be a Poisson distributed random variable with mean $\E[W]=\sum_i p_i$.
For $1 \leq k \leq l$, let $U_k$ and $V_k$ be random variables on the same probability space such that $U_k$ has the same distribution as $W$ and $1+V_k$ has the same distribution as $W$ conditioned on the event $Y_k=1$ (with the convention $V_k=0$ if $\P(Y_k=1)=0$).
Then
\begin{equation*}
d_{TV}(W,Z)\leq \left(1 \wedge \frac{1}{\E[W]}\right)\sum_{i=1}^lp_i\E[\left|U_i-V_i\right|].
\end{equation*}
\end{lemma}
Then, it suffices to exhibit a coupling of $U_i$ and $V_i$ such that $\E[|U_i-V_i|]$ is small.
We refer to \cite{Bartek_notes}, \cite{lindvall2002} or \cite{BarHolJan} for a comprehensive overview of the Chen-Stein method.

We now adapt the Chen-Stein method to the replica-mean-field framework, generalizing the method to the case of a random sum of Bernoulli random variables.

We first give a result that will be an immediate corollary of the lemma we prove afterwards to compare it with Lemma \ref{lem_chenstein_standard}:
\begin{lemma}
\label{lem_cor_rand_chenstein}
Let $L$ be a $\N$-valued random variable such that $\E[L]<\infty$. Let $(Y_i)_{i \in \N}$ be random variables such that for any $l \in \N$, conditionally on the event $\{L=l\}$, $Y_i$ are Bernoulli random variables with respective means $p_i$.
Consider $W=\sum_{i=1}^L Y_i.$Let $Z$ be a Poisson distributed random variable with mean $\E[W]=\E[\sum_i^L p_i]$.
For $k \in \N$, let $U_k$ and $V_k$ be random variables on the same probability space such that for $l \in \N$, conditionally on the event $\{L=l\}$, $U_k$ has the same distribution as $W$ and $1+V_k$ has the same distribution as $W$ conditioned on the event $Y_k=1$ (with the convention $V_k=0$ if $\P(Y_k=1|L=l)=0$).
Then
\begin{equation*}
d_{TV}(W,Z)\leq \left(1 \wedge \frac{0.74}{\E[W]}\right)\E\left[\left|L-\E[L]\right|\right]+\left(1 \wedge \frac{1}{\E[W]}\right)\E\left[\sum_{i=1}^L p_i\E\left[\left|U_i-V_i\right|\big|L\right]\right].
\end{equation*}
\end{lemma}

For our purposes, we will prove a slightly different result with $p_i=\frac{1}{M-1}$ in a vector setting, but it is easy to see that Lemma \ref{lem_cor_rand_chenstein} can be proven in the same way as what follows.

We will now use notation consistent with \eqref{eq_arr_rep}. Since what follows is done with $t \in [0,T]$ fixed, we will additionally write $N_{n,j}([0,t])$ as $N_{n,j}$ in this section, continuing to omit the $M$ superscript to simplify notation. 
\begin{lemma}
\label{lem_poisson_chenstein_bound}
Let $M>1$. Let $(m,i) \in \{1,\ldots,M\}\times \{1,\ldots,K\}$. For $j \in \{1,\ldots,K\}\setminus\{i\}$, let $A_{j \rightarrow (m,i)}=\sum_{n \neq m}\sum_{k=1}^{N_{n,j}}B_{k, (n,j) \rightarrow (m,i)}$ and let $\tilde{A}_{j \rightarrow i}$ be independent Poisson random variables with means $\E[N_{1,j}]$ as in \eqref{eq_limit_eq_nonexch}. Then,
\begin{equation}
\label{eq_chen_poisson_bound}
\begin{split}
d_{TV}(A_{j \rightarrow (m,i)},\tilde{A}_{j \rightarrow i}) &\leq \left(1\wedge\frac{0.74}{\sqrt{\E[N_{1,j}]}}\right)\frac{1}{M-1}\E\left[\left|\sum_{n \neq m}\left(\E[N_{n,j}]-N_{n,j}\right)\right|\right]\\
&+\frac{1}{M-1}\left(1\wedge \frac{1}{\E[N_{1,j}]}\right)\E[N_{1,j}].
\end{split}
\end{equation}
\end{lemma}

A few remarks on this result are in order. First, note that the two terms in the right hand side of \eqref{eq_chen_poisson_bound} are very different in nature. The second term goes to 0 when $M \rightarrow \infty$ due to the moment bound obtained in Lemma \ref{lem_moment_bound}, whereas the first term is the $\mathcal{L}^1$ norm of an empirical mean of centered random variables which are not independent. As such, obtaining the convergence to 0 of that term when $M \rightarrow \infty$  requires proving an $\mathcal{L}^1$ law of large numbers result for non i.i.d. summands, which is not trivial and will be the subject of the next section.
Next, note that the two terms can be heuristically interpreted in the following way: the second term represents a Le Cam-type bound \cite{LeCam86} between a sum of Bernoulli random variables and a sum of Poisson random variables with equal means, in the case where the amount of summands is random. The first term represents the distance between such a random sum of Poisson random variables and a Poisson random variable whose mean is the mean number of summands, similar to \cite{teerapabolarn2014}.

Finally, note that this lemma only provides a bound for fixed $i$ and $j\in \{1,\ldots,K\}\setminus\{i\}: A_{j \rightarrow (m,i)}$ represents the arrivals from nodes $j$ across replicas to node $i$ in replica $m$. Thus Lemma \ref{lem_poisson_chenstein_bound} does not directly give a bound for the approximation of $A_{m,i}$ by $\tilde{A}_{i}=\sum_{j\neq i}\mu_{j\rightarrow i}\tilde{A}_{j \rightarrow i}$. However, since by Lemma \ref{lem_asympt_indep}, we have asymptotic independence, it is natural to expect that the eventual convergence in total variation will also take place for the sum, and we shall see later that it does indeed hold.

We now proceed to the proof of Lemma \ref{lem_poisson_chenstein_bound}.

\begin{proof}
For $B \subset \N$ and $j \in \{1,\ldots,K\}$, let $g_B$ be the solution to the following equation, sometimes referred to as the Stein equation, see \cite{Chen75}:
\begin{equation*}
\E[N_{1,j}]g_B(k+1)-kg_B(k)=\one_B(k)-\P(\tilde{A}_{j \rightarrow i} \in B),
\end{equation*}
for $k \in \N,$ with initial condition $g_B(0)=0$.

As in the proof of Lemma \ref{lem_cond_indep_bern}, let $N =(N_{n,j})_{(n,j)\in (\{1,\ldots,M\}\setminus \{m\})\times(\{1,\ldots,K\}\setminus \{i\})}$. \\
Let $B \subset \N$. 
We have
\begin{equation*}
\begin{split}
&\P(A_{j\rightarrow (m,i)} \in B)-\P(\tilde{A}_{j \rightarrow i}\in B)=\E[\one_{A_{m,i,j\in B}}-\P(\tilde{A}_{j \rightarrow i}\in B)]\\
&=\E\left[\E[N_{1,j}] g_B(A_{j\rightarrow (m,i)}+1)-A_{j\rightarrow (m,i)}g_B(A_{j\rightarrow (m,i)})\right] \textrm{ by the Stein equation} \\
&=\E\left[\E\left[\E[N_{1,j}] g_B(A_{j\rightarrow (m,i)}+1)-\sum_{n \neq m}\sum_{k=1}^{N_{n,j}}B_{k, (n,j) \rightarrow (m,i)}g_B(A_{j\rightarrow (m,i)})\bigg|N\right]\right]\\
&=\E\bigg[\E\bigg[\frac{1}{M-1}\sum_{n\neq m}(\E[N_{n,j}]-N_{n,j}) g_B(A_{j\rightarrow (m,i)}+1)\\
&+\sum_{n \neq m}\sum_{k=1}^{N_{n,j}}\left(\frac{g_B(A_{j\rightarrow (m,i)}+1)}{M-1}-B_{k, (n,j) \rightarrow (m,i)}g_B(A_{j\rightarrow (m,i)})\right)\bigg|N\bigg]\bigg].
\end{split}
\end{equation*}

For all $n \neq m$ and all $1 \leq k \leq N_{n,j}$, let $U_{k,(n,j)\rightarrow (m,i)}$ and $V_{k,(n,j)\rightarrow (m,i)}$ be random variables on the same probability space such that $U_{k,(n,j)\rightarrow (m,i)}\overset{\mathcal{L}}{=}A_{j\rightarrow (m,i)}$ and $\P(V_{k,(n,j)\rightarrow (m,i)}+1 \in \cdot)\overset{\mathcal{L}}{=}\P(A_{j\rightarrow (m,i)} \in \cdot |B_{k, (n,j) \rightarrow (m,i)}=1)$ conditionally on $N$. 

Using Lemma \ref{lem_cond_indep_bern}, we have for all $k$
\begin{equation*}
\E[B_{k, (n,j) \rightarrow (m,i)}|N]=\frac{1}{M-1}.  
\end{equation*}
Therefore, 
\begin{equation*}
\begin{split}
&\P(A_{j\rightarrow (m,i)}\in B)-\P(\tilde{A}_{j \rightarrow i}\in B)=\\
&\E\left[\E\left[\frac{1}{M-1}\sum_{n\neq m}(\E[N_{n,j}]-N_{n,j}) g_B(A_{j\rightarrow (m,i)}+1)\bigg|N\right]\right]\\
&+\frac{1}{M-1}\E\left[\E\left[\sum_{n \neq m}\sum_{k=1}^{N_{n,j}}\left(g_B(U_{k,(n,j)\rightarrow (m,i)}+1)-g_B(V_{k,(n,j)\rightarrow (m,i)}+1)\right)\bigg|N\right]\right].
\end{split}
\end{equation*}

Thus, we have:
\begin{equation*}
\begin{split}
&\left|\P(A_{j\rightarrow (m,i)}\in B)-\P(\tilde{A}_{j \rightarrow i}\in B)\right|\leq \frac{\|g_B\|}{M-1}\E\left[\left|\sum_{n\neq m}\left(\E[N_{n,j}]-N_{n,j}\right)\right|\right]\\
&+\frac{\|\Delta g_B\|}{M-1}\sum_{n \neq m}\E\left[\sum_{k=1}^{N_{n,j}}\E\left[\left|U_{k,(n,j)\rightarrow (m,i)}-V_{k,(n,j)\rightarrow (m,i)}\right| \bigg|N\right]\right],
\end{split}
\end{equation*}
where for a function $f,$ we denote $\|f\|=\sup_{t\in[0,T]}|f(t)|$.
Now, take 
\begin{equation*}
U_{k,(n,j)\rightarrow (m,i)}=A_{j\rightarrow (m,i)} \textrm{ and } V_{k,(n,j)\rightarrow (m,i)}=\sum_{l\neq k}B_{l, (n,j) \rightarrow (m,i)}.   
\end{equation*} Then, 
\begin{equation*}
|U_{k,(n,j)\rightarrow (m,i)}-V_{k,(n,j)\rightarrow (m,i)}|=B_{k, (n,j) \rightarrow (m,i)}.
\end{equation*}
Therefore, using once again Lemma \ref{lem_cond_indep_bern},
\begin{equation*}
\E\left[\left|U_{k,(n,j)\rightarrow (m,i)}-V_{k,(n,j)\rightarrow (m,i)}\right| \big|N\right]=\frac{1}{M-1}.
\end{equation*} 
Moreover, it can be shown (see \cite{BarHolJan}) that $\|g_B\|\leq 1\wedge\frac{0.74}{\E[N_{1,j}]} $ and $\|\Delta g_B\|\leq 1\wedge\frac{1}{\E[N_{1,j}]},$
where for $k \in \N, \Delta g_B(k)=g_B(k+1)-g_B(k).$ Combining this yields \eqref{eq_chen_poisson_bound}.
\end{proof}

\subsection*{Decoupling arrivals and outputs: a fixed point scheme approach}
As we have seen in the previous lemma, for the Poisson approximation to hold, it is sufficient to prove a law of large numbers-type result on the random variables $(N_{n,j})_{n \neq m}$. However, since these random variables themselves depend on the random variables $(A_{j\rightarrow (m,i)})_{m \in \{1,\ldots,M\}}$, a direct proof seems difficult to obtain.

As such, we propose to see Equation \eqref{eq_RMF_LGL_SDE} as the fixed point equation of some function on the space of probability laws on the space of càdlàg trajectories. This fixed point exists and is necessarily unique due to the fact that Equation \eqref{eq_RMF_LGL_SDE} admits a unique solution. The main idea goes as follows: if we endow this space with a metric that makes it complete, in order to prove that the law of large numbers holds at the fixed point, it is sufficient to show that, on one hand, if this law of large numbers holds for a given probability law, it also holds for its image by the function; and that on the other hand, the function's iterates form a Cauchy sequence. This approach is similar to the one developed in \cite{Davydov2022}, where propagation of chaos is proven in discrete time by showing that the one-step transition of the discrete dynamics preserves a triangular law of large numbers.

Our goal in this section is to prove the two aforementioned points. We start by introducing the metric space we will be considering and defining the function on it.

Fix $T \in \R$, and let $D_T$ be the space of càdlàg functions on $[0,T]$ endowed with the Billingsley metric \cite{billingsley1968}: for $x,y \in D_T$, let $$d_{D_T}(x,y)=\inf_{\theta \in \Theta}\max(|||\theta|||,\|x-y\circ \theta\|),$$
where $$\Theta=\{\theta:[0,T]\rightarrow[0,T], \textrm{ s.t. } \theta(0)=0, \theta(T)=T, \textrm{ and } |||\theta|||<\infty \},$$ where $$|||\theta|||=\sup_{s\neq t \in [0,T]}\left|\log\left(\frac{\theta(t)-\theta(s)}{t-s}\right)\right|.$$ Intuitively, $\Theta$ represents all possible "reasonable" time shifts allowing one to minimize the effect of the jumps between the two functions $x$ and $y$, where "reasonable" means that all slopes of $\theta$ are close to 1.

We denote by $d_{D_T,U}$ the uniform metric on $D_T$: for $x,y \in D_T$, \begin{equation*}
d_{D_T,U}(x,y)=\|x-y\|.
\end{equation*} 
Note that we have for all $x,y \in D_T, d_{D_T}(x,y)\leq d_{D_T,U}(x,y)$, since the uniform metric corresponds precisely to the case where $\theta$ is the identity function.

Let $\mathcal{P}(D_T)$ be the space of probability measures on $D_T$. We endow it with the Kantorovitch metric \cite{Kanto42} (sometimes also known as the Wasserstein distance or the earth mover's distance): for $\mu,\nu \in \mathcal{P}(D_T)$, let 
\begin{equation*}
K_T(\mu,\nu)=\inf_{\Pi \in D_T \times D_T}\E[ d_{D_T}(x,y)],
\end{equation*}
where $\Pi$ is a coupling s.t. $x \overset{\mathcal{L}}{=} \mu$ and $y\overset{\mathcal{L}}{=} \nu$. 

Finally, we fix $K,M \in \N$ and consider the space $(\mathcal{P}(D_T))^{MK}$ endowed with the 1-norm metric: for $\mu,\nu \in (\mathcal{P}(D_T))^{MK}$, let $$K_T^{MK}(\mu,\nu)=\sum_{m=1}^M\sum_{i=1}^K K_T(\mu_{m,i},\nu_{m,i}).$$

It is known that $(D_T,d_{D_T})$ is a complete separable metric space, see \cite{billingsley1968}, and thus
that $(\mathcal{P}(D_T),K_T)$ and $(\mathcal{P}(D_T))^{MK},K_T^{MK})$ are as well, see \cite{BogKol2012}.

We will also need to consider $\mathcal{P}(D_T)$ endowed with a Kantorovitch metric based on the uniform metric: we introduce for $\mu,\nu \in \mathcal{P}(D_T)$,  $$K_{T,U}(\mu,\nu)=\inf_{\Pi \in D_T \times D_T} \E[d_{D_T,U}(x,y)],$$
where $\Pi$ is a coupling s.t. $x \overset{\mathcal{L}}{=} \mu$ and $y\overset{\mathcal{L}}{=} \nu$. We also introduce its product version $K_{T,U}^{MK}$ defined analogously to above. Note that even though $(D_T,d_{D_T,U})$ is a complete metric space, it is not separable, therefore $(\mathcal{P}(D_T),K_{T,U})$ is not a priori a complete metric space.

We now define the following mapping:

\begin{align*}
  \Phi \colon (\mathcal{P}(D_T))^{MK} &\to (\mathcal{P}(D_T))^{MK}\\
  (\mathcal{L}(M))&\mapsto \Phi(\mathcal{L}(M)),
\end{align*}
\\
where for all $(m,i)\in \{1,\ldots,M\}\times\{1,\ldots,K\}, \Phi(\mathcal{L}(M))_{m,i}$ is the law of the stochastic intensity $\lambda^{\Phi}_{m,i}$ of a point process $N^{\Phi}_{m,i}$ such that $\lambda^{\Phi}$ is the solution of the stochastic differential equation 
\begin{equation}
\label{eq_phi_def}
\begin{split}
\lambda^{\Phi}_{m,i}(t)&=\lambda^{\Phi}_{m,i}(0)+\sum_{j \neq i}\mu_{j\rightarrow i}\sum_{n \neq m}\int_0^t\one_{\{V^M_{(n,j)\rightarrow i}(s)=m\}}M_{n,j}(\diff s)\\
&+\int_0^t \left(r_i-\lambda^{\Phi}_{m,i}(s)\right)N^{\Phi}_{m,i}(\diff s),
\end{split}
\end{equation}
where $(\lambda^{\Phi}_{m,i}(0))$ are random variables verifying Assumption \ref{Ass_2}. Note that we will exclusively apply the mapping $\Phi$ to laws of stochastic intensities of point processes, the image of which by $\Phi$ are also by definition of $\Phi$ laws of stochastic intensities of point processes.

We formalize the law of large numbers we aim to prove as follows:
\begin{definition}
Let $M \in \N$. Let $(X^M_n)_{1 \leq n\leq M}$ be exchangeable random variables with finite expectation. We say they satisfy an $\mathcal{L}^1$ triangular law of large numbers, which we denote $\TLLN(X^M_n)$, if when $M \rightarrow \infty,$
\begin{equation}
\label{eq_TLLN_cont}
\E\left[\left|\frac{1}{M-1}\sum_{n=1}^M (X^M_n-\E[X^M_n])\right|\right]\rightarrow 0
\end{equation}
and
\begin{equation}
\label{eq_TLLN_cont_2}
X^M_n\Rightarrow \tilde{X},
\end{equation}
where the convergence takes place in distribution.
\end{definition}

From \eqref{eq_chen_poisson_bound}, we know that if the triangular law of large numbers holds for the fixed point of $\Phi$, it allows for convergence in total variation of arrivals across replicas from a given neuron $j$ to a given neuron $i$ to a Poisson random variable. As such, our aim here is twofold:

\begin{enumerate}
\item Show that for all $(m,i)\in \{1,\ldots,M\}\times\{1,\ldots,K\}, \TLLN(\overline{N_{m,i}}([0,t]))$ implies \\$\TLLN(\Phi(\overline{N_{m,i}}([0,t])))$;
\item Show that $(\Phi^l)_{l \in \N^*}$ is a Cauchy sequence that converges to the fixed point, where $\Phi^l$ is the $l$-th iterate of $\Phi: \Phi^l=\Phi \circ \Phi \circ \ldots \circ \Phi$ $l$ times.
\end{enumerate}
Since we can choose $\overline{N_{m,i}}([0,t])$ to be i.i.d. to ensure that there exist inputs for which $\TLLN$ holds, this will allow us to propagate the property and show that $\TLLN$ holds at the fixed point as well.

We will start by proving a lemma that will be key for the second point:

\begin{lemma}
\label{lem_sznitbound}
There exists $T>0$ such that for $\rho, \nu \in (\mathcal{P}(D_T))^{MK}$ that are laws of stochastic intensities of point processes, there exists a constant $C_T>0$ such that
\begin{equation}
\label{eq_sznitbound}
K_{T,U}^{MK}(\Phi(\rho),\Phi(\nu))\leq C_T \int_0^T K_{t,U}^{MK}(\rho, \nu)\diff t.
\end{equation}
\end{lemma}
\begin{proof}
Let $T>0.$ Let $t\in[0,T]$. Fix $(m,i) \in \{1,\ldots,M\}\times\{1,\ldots,K\}$. Let $N^{\rho}$ (resp. $N^{\nu},N^{\Phi(\rho)},N^{\Phi(\nu)})$ be a point process admitting $\rho$ (resp. $\nu,\Phi(\rho),\Phi(\nu)$) as a stochastic intensity.
We have
\begin{equation*}
\begin{split}
\Phi(\rho)_{m,i}(t)-\Phi(\nu)_{m,i}(t)&=\sum_{j\neq i}\mu_{j\rightarrow i}\sum_{n \neq m}\left(\int_0^t\one_{\{V^M_{(n,j)\rightarrow i}(s)=m\}}(N^{\rho}_{n,j}(\diff s)-N^{\nu}_{n,j}(\diff s))\right)\\
&+\int_0^t (r_i-\Phi(\rho)_{m,i}(s))N^{\Phi(\rho)}_{m,i}(\diff s)\\
&-\int_0^t (r_i-\Phi(\nu)_{m,i}(s))N^{\Phi(\nu)}_{m,i}(\diff s).
\end{split}
\end{equation*}
Let $(\hat{N}_{m,i})_{(m,i) \in \{1,\ldots,M\}\times\{1,\ldots,K\}}$ be independent Poisson point processes with intensity 1 on $[0,T]\times\R^+$ Using the Poisson embedding construction, we can write
\begin{equation*}
\begin{split}
&\Phi(\rho)_{m,i}(t)-\Phi(\nu)_{m,i}(t)=\\
&\sum_{j\neq i}\mu_{j\rightarrow i}\sum_{n \neq m}\int_0^t\int_0^{+\infty}\one_{\{V^M_{(n,j)\rightarrow i}(s)=m\}}(\one_{\{u\leq \rho_{n,j}(s)\}}-\one_{\{u\leq \nu_{n,j}(s)\}})\hat{N}_{n,j}(\diff s\diff u)\\
&+r_i\int_0^t\int_0^{+\infty}(\one_{\{ u\leq \Phi(\rho)_{m,i}(s)\}}-\one_{\{ u\leq \Phi(\nu)_{m,i}(s)\}})\hat{N}_{m,i}(\diff s\diff u)\\
&+\int_0^{t}\int_0^{+\infty} (\Phi(\nu)_{m,i}(s)\one_{\{u\leq (\Phi(\nu)_{m,i}(s))\}}-\Phi(\rho)_{m,i}(s))\one_{\{u\leq (\Phi(\rho)_{m,i}(s))\}}\hat{N}_{m,i}(\diff s\diff u).
\end{split}
\end{equation*}

Therefore, we have
\begin{equation*}
\begin{split}
&\left|\Phi(\rho)_{m,i}(t)-\Phi(\nu)_{m,i}(t)\right|\leq\\
&\sum_{j\neq i}\mu_{j\rightarrow i}\sum_{n \neq m}\int_0^t\int_0^{+\infty}\one_{\{V^M_{(n,j)\rightarrow i}(s)=m\}}\one_{\{u\leq \sup_{z \in [0,s]}|\rho_{n,j}(z)-\nu_{n,j}(z)|\}}\hat{N}_{n,j}(\diff s\diff u)\\
&+r_i\int_0^t\int_0^{+\infty} \one_{\{u\leq \sup_{z \in [0,s]}|\Phi(\rho)_{m,i}(z)-\Phi(\nu)_{m,i}(z)|\}}\hat{N}_{m,i}(\diff s\diff u)\\
&+\int_0^t\int_0^{+\infty}\sup_{z \in [0,s]}|\Phi(\rho)_{m,i}(z)-\Phi(\nu)_{m,i}(z)|\one_{\{u\leq \Phi(\rho)_{m,i}(s)\wedge\Phi(\nu)_{m,i}(s)\}}\hat{N}_{m,i}(\diff s\diff u)\\
&+\int_0^t\int_0^{+\infty}|\Phi(\rho)_{m,i}(s)\vee\Phi(\nu)_{m,i}(s)|\one_{\{u\leq \sup_{z \in [0,s]}|\Phi(\rho)_{m,i}(z)-\Phi(\nu)_{m,i}(z)|\}}\hat{N}_{m,i}(\diff s\diff u).
\end{split}
\end{equation*}
Taking the expectation, using the property of stochastic intensity and proceeding as before to obtain the $\frac{1}{M-1}$ from the routing indicators, we get
\begin{equation*}
\begin{split}
&\E\left[\sup_{t \in [0,T]}\left|\Phi(\rho)_{m,i}(t)-\Phi(\nu)_{m,i}(t)\right|\right]\leq \\
&\frac{1}{M-1}\sum_{j\neq i}\mu_{j\rightarrow i}\sum_{n \neq m}\int_0^T\E\left[\sup_{z\in[0,s]}\left|\rho_{n,j}(z)-\nu_{n,j}(z)\right|\right]\diff s\\
&+r_i\int_0^T\E\left[\sup_{z \in [0,s]}\left|\Phi(\rho)_{m,i}(z)-\Phi(\nu)_{m,i}(z)\right|\right]\diff s\\
&+\int_0^T\E\left[\sup_{z \in [0,s]}\left|\Phi(\rho)_{m,i}(z)-\Phi(\nu)_{m,i}(z)\right|\left(\Phi(\rho)_{m,i}(s)\wedge \Phi(\nu)_{m,i}(s)\right)\right]\diff s\\
&+\int_0^T\E\left[\sup_{z \in [0,s]}\left|\Phi(\rho)_{m,i}(z)-\Phi(\nu)_{m,i}(z)\right|\left(\Phi(\rho)_{m,i}(s)\vee \Phi(\nu)_{m,i}(s)\right)\right]\diff s.
\end{split}
\end{equation*}

Denote $||\mu||=\max_{i,j}\mu_{j\rightarrow i}$. We then have
\begin{equation*}
\begin{split}
&\frac{1}{M-1}\sum_{j\neq i}\mu_{j\rightarrow i}\sum_{n \neq m}\int_0^T\E\left[\sup_{z\in[0,s]}\left|\rho_{n,j}(z)-\nu_{n,j}(z)\right|\right]\diff s\leq \\
&||\mu||\sum_{j=1}^K\sum_{n=1}^M\int_0^T d_{D_s,U}(\rho_{n,j},\nu_{n,j})\diff s,
\end{split}
\end{equation*}
from which we immediately get by definition of $K_{T,U}^{MK}$
\begin{equation}
\label{eq_contracterm}
\frac{1}{M-1}\sum_{j\neq i}\mu_{j\rightarrow i}\sum_{n \neq m}\int_0^T\E[\sup_{z\in[0,s]}|\rho_{n,j}(z)-\nu_{n,j}(z)|]\diff s\leq ||\mu||\int_0^T K_{s,U}^{MK}(\rho,\nu)\diff s.
\end{equation}

Let $C>0$. As before, let $A_C([0,T])=\{(\omega,t) \in \Omega\times [0,T], \Phi(\rho)_{m,i}(t)\vee \Phi(\nu)_{m,i}(t)>C\}$.
Using the exact same reasoning as in Lemma \ref{lem_moment_equality}, we have for small enough $T$ the existence of a constant $K_T>0$ such that
\begin{equation*}
\begin{split}
&\int_0^t\E[\sup_{z \in [0,s]}|\Phi(\rho)_{m,i}(z)-\Phi(\nu)_{m,i}(z)|(\Phi(\rho)_{m,i}(s)\wedge \Phi(\nu)_{m,i}(s))]\diff s\\
&\leq C\int_0^t \E\left[\sup_{z\in[0,s]}\left|\Phi(\rho)_{m,i}(z)-\Phi(\nu)_{m,i}(z)\right|\right]\diff s+K_Te^{-3CT}.
\end{split}
\end{equation*}
Plugging in \eqref{eq_contracterm} and applying the same reasoning as above to the last integral term, we get the existence of a constant $K'_T>0$ such that
\begin{equation*}
\begin{split}
&\E\left[\sup_{t \in [0,T]}\left|\Phi(\rho)_{m,i}(t)-\Phi(\nu)_{m,i}(t)\right|\right]\leq ||\mu||\int_0^T K_{s,U}^{MK}(\rho,\nu)\diff s\\
&+\left(r_i+2C\right)\int_0^T\E\left[\sup_{z \in [0,s]}\left|\Phi(\rho)_{m,i}(z)-\Phi(\nu)_{m,i}(z)\right|\right]\diff s\\
&+\left(K_T+K'_T\right)e^{-3CT}.
\end{split}
\end{equation*}
Applying Grönwall's lemma, we get
\begin{equation*}
\begin{split}
&\E\left[\sup_{t \in [0,T]}\left|\Phi(\rho)_{m,i}(t)-\Phi(\nu)_{m,i}(t)\right|\right]\leq \\
&\left(||\mu||\int_0^T K_{s,U}^{MK}(\rho,\nu)\diff s + \left(K_T+K'_T\right)e^{-3CT}\right)e^{(r_i+2C)T}.
\end{split}
\end{equation*}
For any $\eps>0$, we can choose $C>0$ such that
\begin{equation*}
\E\left[\sup_{t \in [0,T]}\left|\Phi(\rho)_{m,i}(t)-\Phi(\nu)_{m,i}(t)\right|\right]\leq \left(||\mu||\int_0^T K_{s,U}^{MK}(\rho,\nu)\diff s\right)e^{(r_i+2C)T}+\eps.
\end{equation*}
Letting $\eps$ go to 0 and taking the sum over all coordinates and the infimum across all couplings, we get the result.
\end{proof}

As previously mentioned, we need to prove convergence of the sequence of iterates of $\Phi$ to the fixed point of $\Phi$ to prove the triangular law of large numbers. We will now derive this from Lemma \ref{lem_sznitbound}.

\begin{lemma}
\label{lem_iter_conv_cFIAP}
Let $\rho \in (\mathcal{P}(D_T))^{MK}$ be the law of the stochastic intensity of a point process.
The sequence $(\Phi^l(\rho))_{l \in \N^*}$ of iterates of the function $\Phi$ is a Cauchy sequence. Moreover, it converges to the unique fixed point of $\Phi$.
\end{lemma}

\begin{proof}
Let $\rho \in (\mathcal{P}(D_T))^{MK}$ be the law of the stochastic intensity of a point process.
By immediate induction, from \eqref{eq_sznitbound}, we have, for all $l \in \N^*,$
\begin{equation*}
K_{T,U}^{MK}(\Phi^{l+1}(\rho),\Phi^l(\rho))\leq C_T^l \frac{T^l}{l!} K_{T,U}^{MK}(\Phi(\rho), \rho).
\end{equation*}
This in turn implies that for any $p<q \in \N^*,$
\begin{equation}
\label{eq_cauchy_bound}
K_{T,U}^{MK}(\Phi^{p}(\rho),\Phi^q(\rho))\leq \sum_{l=p}^{q-1}\frac{C_T T^l}{l!} K_{T,U}^{MK}(\Phi(\rho), \rho).
\end{equation}
Since the series on the right hand side is converging, it proves that the sequence $(\Phi^l)_{l \in \N^*}$ is a Cauchy sequence for the $K_{T,U}^{MK}$ metric.
The space $(\mathcal{P}(D_T))^{MK},K_{T,U}^{MK})$ is not complete. However, since for any $\mu, \nu \in D_T, d_{D_T}(\mu,\nu)\leq d_{D_T,U}(\mu,\nu),$ it follows that $(\Phi^l)_{l \in \N^*}$ is a Cauchy sequence for the $K_{T}^{MK}$ metric  as well.
By completeness of $(\mathcal{P}(D_T))^{MK},K_T^{MK})$, $(\Phi^l)_{l \in \N^*}$ converges to the unique fixed point of $\Phi$.
\end{proof}

All that remains is proving that the triangular law of large numbers is carried over by the function $\Phi$, namely, that if we have some input $X$ that verifies $\TLLN(X)$, then we have $\TLLN(\Phi(X)).$

To do so, the key lemma will be the following law of large numbers:

\begin{lemma}
\label{lem_tlln}
Let $M \in \N^*$. Let $(X_1^M, \ldots, X^M_M)$ be $M$-exchangeable centered random variables with finite exponential moments. Suppose that for any $N \in \N^*$, $(X^M_1,\ldots, X^M_N) \overset{\mathcal{L}}{\rightarrow} (\tilde{X}_1,\ldots, \tilde{X}_N)$ when $M \rightarrow \infty$, where $(\tilde{X}_i)_{i \in \N^*}$ are i.i.d. random variables and the convergence takes place in distribution. Then
\begin{equation}
\label{eq_tlln_lem}
\E\left[\left|\frac{1}{M}\sum_{n=1}^M X^M_n\right|\right] \rightarrow 0
\end{equation}
when $M \rightarrow \infty$.
\end{lemma}

\begin{proof}
Let $$U_M=\frac{1}{M}\sum_{n=1}^M X^M_n.$$ Note that $\E[U_M]=0.$
We have
\begin{equation*}
\E[U_M^2]=\frac{1}{M^2}\E\left[\sum_{n=1}^M(X^M_n)^2+\sum_{m \neq n}X^M_m X^M_n\right].
\end{equation*}
Since the exponential moments of $(X^M_n)$ are bounded and they all converge in distribution, we have by asymptotic independence that for any $m,n$,
\begin{equation*}
\E[X^M_m X^M_n]\rightarrow \E[\tilde{X}_m \tilde{X}_n]=\E[\tilde{X}_m]\E[\tilde{X}_n]=0
\end{equation*} 
when $M \rightarrow \infty$.

Therefore, for $\eps >0$, for large enough $M$, each of the terms of the corresponding sums in the equation above is smaller than $\eps$. Combining this with exchangeability, we get the existence of a positive constant $C$ s.t. for large enough $M$,
\begin{equation*}
\E[U_M^2] \leq \frac{1}{M^2}(CM+M(M-1)\eps).
\end{equation*}
Now, applying Chebychev's inequality, for any $\delta >0,$
\begin{equation*}
\P(|U_M-\E[U_M]|>\delta) \leq \frac{\E[U_M^2]}{\delta^2}\leq \frac{C}{M\delta^2}+\frac{M(M-1)\eps}{M^2\delta^2}.
\end{equation*}
This gives convergence in probability of $U_M$ to 0 when $M\rightarrow\infty$.

Since in addition the second moments are uniformly bounded, $\mathcal{L}^1$ convergence follows.
\end{proof}

Now, we can finally prove the following lemma which is the last step needed to prove the main theorem:

\begin{lemma}
\label{lem_tlln_Phi}
Let $(\overline{N_{m,i}})$ be point processes on $[0,T]$ with finite exponential moments. Let $t \in [0,T].$ Suppose $\TLLN((\overline{N_{m,i}}([0,t]))$ holds. Then, \\$\TLLN(\Phi((\overline{N_{m,i}}([0,t])))$ holds as well.
\end{lemma}

\begin{proof}
Suppose $\TLLN((\overline{N_{m,i}}([0,t]))$ holds. For $(m,i)\in \{1,\ldots,M\}\times\{1,\ldots,K\}$, let $\overline{\lambda_{m,i}}$ be the stochastic intensity of the process $\overline{N_{m,i}}$.
We write for all $t \in [0,T]$,
\begin{equation*}
\lambda_{m,i}(t)=\lambda_{m,i}(0)+\sum_{j \neq i}\mu_{j\rightarrow i}\overline{A_{j\rightarrow (m,i)}}(t)+\int_0^t (r_i-\lambda_{m,i}(s))N_{m,i}(\diff s),
\end{equation*}
where $(\lambda_{m,i}(0))$ verifies Assumption \ref{Ass_2} and 
\begin{equation*}
\overline{A_{j\rightarrow (m,i)}}(t)=\sum_{n \neq m}\int_0^t\one_{\{V^M_{(n,j)\rightarrow i}(s)=m\}}\overline{N_{n,j}}(\diff s).
\end{equation*}

Analogously to \eqref{eq_chen_poisson_bound}, we have
\begin{equation*}
\begin{split}
&d_{TV}(\overline{A_{j\rightarrow (m,i)}}(t),\tilde{A}_{j \rightarrow i}(t)) \leq\\ &\left(1\wedge\frac{0.74}{\sqrt{\E[\overline{N_{1,j}}\left([0,T]\right)]}}\right)\frac{1}{M-1}\E[|\sum_{n \neq m}\left(\E[\overline{N_{n,j}}\left([0,T]\right)]-\overline{N_{n,j}}\left([0,T]\right)\right)|]\\
&+\frac{1}{M-1}\left(1\wedge \frac{1}{\E[\overline{N_{1,j}}\left([0,T]\right)]}\right)\E[\overline{N_{1,j}}\left([0,T]\right)].
\end{split}
\end{equation*}
where $\tilde{A}_{j \rightarrow i}$ are independent Poisson random variables with mean $\E[N_{1,j}([0,T])]$.

As such, $\TLLN((\overline{N_{m,i}}([0,t]))$ implies convergence in total variation of $(\overline{A_{j\rightarrow (m,i)}})$ to independent random variables $(\tilde{A}_{j \rightarrow i})$. We will now show that this implies convergence in total variation of $\sum_{j \neq i}\mu_{j\rightarrow i}\overline{A_{j\rightarrow (m,i)}}.$ 

Denote as before $\overline{N}=(\overline{N}_{n,j})_{n\neq m, j \neq i}.$ Let $q\in \N^{(M-1)(K-1)}.$ Let $B_1,B_2\in \mathcal{B}(\R^+).$ Let $l_1\neq l_2 \in \{1,\ldots,K\}\setminus \{i\}.$
Then, using the total probability formula, we have
\begin{equation*}
\begin{split}
&\P(\overline{A}_{l_1\rightarrow(m,i)}\in B_1,\overline{A}_{l_2\rightarrow(m,i)}\in B_2)=\\
&\sum_q \P(\overline{A}_{l_1\rightarrow(m,i)}\in B_1,\overline{A}_{l_2\rightarrow(m,i)}\in B_2|\overline{N}=q)\P(\overline{N}=q).
\end{split}
\end{equation*}
Using Lemma \ref{lem_cond_indep_bern}, by conditional independence, we have
\begin{equation*}
\begin{split}
&\P(\overline{A}_{l_1\rightarrow(m,i)}\in B_1,\overline{A}_{l_2\rightarrow(m,i)}\in B_2)=\\
&\sum_q \P(\overline{A}_{l_1\rightarrow(m,i)}\in B_1|\overline{N}=q)\P(\overline{A}_{l_2\rightarrow(m,i)}\in B_2|\overline{N}=q)\P(\overline{N}=q).
\end{split}
\end{equation*}
Using the same reasoning as in Lemma \ref{lem_poisson_chenstein_bound}, we have that: 
\begin{equation*}
\P(\overline{A}_{l_1\rightarrow(m,i)}\in B_1|\overline{N}=q)\rightarrow \P(\tilde{A}_{l_1\rightarrow i}\in B_1|\overline{N}=q)=\P(\tilde{A}_{l_1\rightarrow i}\in B_1),
\end{equation*} 
\begin{equation*}
\P(\overline{A}_{l_2\rightarrow(m,i)}\in B_2|\overline{N}=q)\rightarrow \P(\tilde{A}_{l_2\rightarrow i}\in B_2|\overline{N}=q)=\P(\tilde{A}_{l_2\rightarrow i}\in B_2)
\end{equation*} and 
\begin{equation*}
\P(\overline{N}=q)\rightarrow\P(\overline{\tilde{N}}=q).
\end{equation*}
By dominated convergence,
\begin{equation*}
\begin{split}
\P(\overline{A}_{l_1\rightarrow(m,i)}\in B_1,\overline{A}_{l_2\rightarrow(m,i)}\in B_2)&\rightarrow \sum_q \P(\tilde{A}_{l_1\rightarrow i}\in B_1)\P(\tilde{A}_{l_2\rightarrow i}\in B_2)\P(\overline{N}=q)\\
&=\P(\tilde{A}_{l_1\rightarrow(m,i)}\in B_1,\tilde{A}_{l_2\rightarrow(m,i)}\in B_2).
\end{split}
\end{equation*}
This implies convergence of $\sum_j \mu_{j\rightarrow i}\overline{A}_{j\rightarrow(m,i)}$.
Finally, the mapping theorem implies convergence in total variation of $\lambda_{m,i}(t)$ when $M \rightarrow \infty$. All conditions of Lemma \ref{lem_tlln} are thus satisfied. Applying it completes the proof.
\end{proof}
Thus, we can now state the result that we were aiming to prove:
\begin{lemma}
Denote by $(N_{m,i})$ the point processes of the $M$-replica RMF dynamics \eqref{eq_RMF_LGL_SDE} that are the fixed point of $\Phi$. Then $\TLLN((N_{m,i}([0,T])))$ holds.
\end{lemma}
\begin{proof}
Let $(\overline{N_{n,j}})$ be random variables satisfying $\TLLN(\overline{N_{n,j}}).$ Let us first write out equalities and justify them afterwards. We have
\begin{equation*}
\begin{split}
&\lim_{M \rightarrow\infty}\frac{1}{M-1}\E\left[\left|\sum_{n \neq m}\left(\E[N_{n,j}([0,T])]-N_{n,j}([0,T])\right)\right|\right]\\
&=\lim_{M \rightarrow\infty}\frac{1}{M-1}\E\left[\left|\sum_{n \neq m}\left(\E[\lim_{l \rightarrow \infty}\Phi^l(\overline{N_{n,j}}([0,T]))]-\lim_{l \rightarrow \infty}\Phi^l(\overline{N_{n,j}}([0,T]))\right)\right|\right]\\
&=\lim_{M \rightarrow\infty}\lim_{l \rightarrow \infty}\frac{1}{M-1}\E\left[\left|\sum_{n \neq m}\left(\E[\Phi^l(\overline{N_{n,j}}([0,T]))]-\Phi^l(\overline{N_{n,j}}([0,T]))\right)\right|\right]\\
&=\lim_{l \rightarrow \infty}\lim_{M \rightarrow\infty}\frac{1}{M-1}\E\left[\left|\sum_{n \neq m}\left(\E[\Phi^l(\overline{N_{n,j}}([0,T]))]-\Phi^l(\overline{N_{n,j}}([0,T]))\right)\right|\right]\\
&=0.
\end{split}
\end{equation*}
The first equality is given by Lemma \ref{lem_iter_conv_cFIAP}. To obtain the second equality, we use the dominated convergence theorem and the fact that all moments all uniformly bounded through Lemma \ref{lem_moment_bound} (note that initial conditions are fixed in the definition of $\Phi$ and are chosen to verify Assumption \ref{Ass_2}).
To justify the third equality, note that from \eqref{eq_cauchy_bound}, using Lemma \ref{lem_moment_bound} to obtain again a uniform bound of the moments, we get that the Cauchy sequence of iterates of $\Phi$ verifies the uniform Cauchy criterion and thus converges uniformly to the fixed point, which in turn allows for the exchange of limits in $M$ and $l$.
The last equality stems directly from Lemma \ref{lem_tlln_Phi}.
\end{proof}
\section{Tightness and convergence on \texorpdfstring{$\R$}{R}}
\label{sec_Rel_comp_RMF_LGL}

The goal of this section is to generalize the main result of the paper. In Theorem \ref{th_Poisson_CV}, we proved weak convergence of the replica-mean-field processes on compacts of $\R^+$. We now prove weak convergence on $\R^+$. One motivation for doing so is that the results on the Galves-Löcherbach replica-mean-field model in the paper \cite{Baccelli_2019} by Baccelli and Taillefumier assumes that the Poisson Hypothesis holds at the limit in the stationary regime. As such, the following result provides the missing rigorous justification, albeit for a slightly simplified model due to Assumption \ref{Ass_1}.

\begin{theorem}
\label{Th_rel_comp_RMF_LGL}
Let $K,M \geq 2.$ For all $m \in \{1,...,M\}$ and $i \in \{1,...,K\}$, the process $\lambda_{m,i}$ weakly converges in the Skorokhod space of càdlàg functions on $\R^+$.
\end{theorem}

Recall the following tightness criterion due to Aldous \cite{Aldous78}:

\begin{Prop}
Let $(\Omega, \mathcal{F}, (\mathcal{F}_t^{(n)}, t \in \mathbb{R}^+), \P)$ be a probability space, let $X^{(n)}$ be adapted càdlàg processes. If for all $T>0, (\mathcal{L}(||X^{(n)}||))$ is tight on $[0,T]$, and if for all $ \epsilon >0$, for all $\epsilon ' >0$, there exists $\delta\in (0,T]$ such that 
\begin{equation*}
\limsup_{n\rightarrow +\infty} \sup_{ (\substack{S_1,S_2 \in \mathcal{F}_t^{(n)}) \text{ such that} \\ S_1 \leq S_2 \leq (S_1+\delta) \wedge T}} \P(|X_{S_1}^{(n)}-X_{S_2}^{(n)}| > \epsilon ') \leq \epsilon,
\end{equation*}
then $(\mathcal{L}(X^{(n)}))$ is tight on the space of càdlàg functions on $\R^+$.
\end{Prop}

We will also require the following inequality on martingales:
\begin{lemma}
\label{lem_nonneg_mart}
Let $(X_t)_{t\in [0,T]}$ be a nonnegative $(\mathcal{F}_t)$-martingale, let $S_1,S_2$ be two stopping times such that $S_1 \leq S_2 \leq S_1+\delta$. Then
\begin{equation}
\label{eq_mart_major}
\E\left[\int_{S_1}^{S_2}X_s \diff s\right] \leq 2\delta \sqrt{\E[X_T^2]}.
\end{equation}
\end{lemma}
\begin{proof}
We have
\begin{equation*}
\begin{split}
\E\left[\int_{S_1}^{S_2}X_s \diff s\right] &\leq \delta\E[\sup_{s\in [0,T]} X_s] \\
&\leq \delta \sqrt{\E[\sup_{s\in [0,T]} X_s^2]} \\
&\leq 2 \delta \sqrt{\E[X_T^2]} \textrm{ (by the Doob inequality) }
\end{split}
\end{equation*}
\end{proof}
We are now ready to prove Theorem \ref{Th_rel_comp_RMF_LGL}.

\begin{proof}
We use the Aldous criterion mentioned above to prove tightness.
Let $T>0$.
Let $\epsilon '>0, \delta >0, S_1,S_2$ two stopping times such that $S_1 \leq S_2 \leq (S_1+\delta) \wedge T$.

Using the Markov inequality, the fact that $r_i-\lambda_{m,i}$ is negative, and the property of stochastic intensity, we have:
\begin{equation*}
\begin{split}
&\P\left(|\lambda_{m,i}(S_2)-\lambda_{m,i}(S_1)|>\epsilon'\right) =\\
&\P\left(|\sum_{j \neq i}\mu_{j\rightarrow i}\sum_{n\neq m}\int_{S_1}^{S_2} \one_{\{V_{(n,j)\rightarrow i}(s)=m\}} N_{n,j}(\diff s)+\int_{S_1}^{S_2} \left(r_i-\lambda_{m,i}(s)\right) N_{m,i}(\diff s)|>\epsilon'\right) \\
 &\leq \frac{1}{\epsilon'}\E\left[\sum_{j \neq i}\mu_{j\rightarrow i}\sum_{n\neq m}\int_{S_1}^{S_2}\one_{\{V_{(n,j)\rightarrow i}(s)=m\}} N_{n,j}(\diff s)+\int_{S_1}^{S_2}\left(\lambda_{m,i}(s)-r_i\right)  N_{m,i}(\diff s) \right]\\
 &\leq \frac{1}{\epsilon'}\E\left[\int_{S_1}^{S_2}\sum_{n\neq m}\sum_{j \neq i}\mu_{j\rightarrow i}\one_{\{V_{(n,j)\rightarrow i}(s)=m\}}\lambda_{n,j}(s)+\left(\lambda_{m,i}^2(s)-r_i\lambda_{m,i}(s)\right)\diff s\right] \\
 &\leq \frac{1}{\epsilon'}\E\left[\int_{S_1}^{S_2}\sum_{j \neq i}\mu_{j\rightarrow i}\lambda_{1,j}(s)+\lambda_{m,i}^2(s)-r_i\lambda_{m,i}(s) \diff s\right],\\
\end{split}
\end{equation*}

Since $\lambda_{m,i}(s)$ is non-negative for all $s \in [0,T]$, we can write
\begin{equation*}
\P\left(|\lambda_{m,i}(S_2)-\lambda_{m,i}(S_1)|>\epsilon'\right)\leq \frac{1}{\epsilon'}\E\left[\int_{S_1}^{S_2}\sum_{j \neq i}\mu_{j\rightarrow i}\lambda_{1,j}(s)+\lambda_{m,i}^2(s)\diff s\right].
\end{equation*}
For all $(m,i)\in \{1,...,M\}\times\{1,...,K\}$, let 
\begin{equation*}
c_{m,i}(t)=\sum_{n \neq m} \sum_{j \neq i} \int_0^t \one_{\{V_{(n,j)\rightarrow i}(s)=m\}} \lambda_{n,j}(s) \diff s+\int_0^t \left(r_i-\lambda_{m,i}(s)\right) \lambda_{m,i}(s) \diff s,
\end{equation*}
and
\begin{equation*}
\begin{split}
d_{m,i}(t)&=\sum_{n \neq m} \sum_{j \neq i} \int_0^t \one_{\{V_{(n,j)\rightarrow i}(s)=m\}} \left(2\lambda_{n,j}(s)+1\right)\lambda_{n,j}(s) \diff s\\
&+\int_0^t \left(r_i^2-\lambda_{m,i}^2(s)\right)\lambda_{m,i}(s) \diff s.
\end{split}
\end{equation*} 
By the property of stochastic intensity, $\rho_{m,i}(t)=\lambda_{m,i}(t)-c_{m,i}(t)$ and $\nu_{m,i}(t)=\lambda_{m,i}^2(t)-d_{m,i}(t)$
are $(\mathcal{F}_t)$-martingales.

Therefore, we can write
\begin{equation}
\label{eq_aldous_bound_1}
\begin{split}
&\P\left(|\lambda_{m,i}(S_2)-\lambda_{m,i}(S_1)|>\epsilon'\right) \leq\\
&\frac{1}{\epsilon'}\E\left[\int_{S_1}^{S_2}\sum_{j\neq i}\mu_{j\rightarrow i}\rho_{1,j}(s)+\nu_{m,i}(s)+\sum_{j \neq i}\mu_{j\rightarrow i}c_{1,j}(s)+d_{m,i}(s) \diff s\right].
\end{split}
\end{equation}

We now bound separately the martingale part and the rest of the right-hand-side expression in \eqref{eq_aldous_bound_1}. 

Using Lemma \ref{lem_nonneg_mart}, we have that
\begin{equation*}
\E\left[\int_{S_1}^{S_2}\left(\sum_{j\neq i}\mu_{j\rightarrow i}\rho_{1,j}(s)+\nu_{m,i}(s)\right)\diff s\right]\leq 2\delta\sqrt{\E[(\sum_{j\neq i}\mu_{j\rightarrow i}\rho_{1,j}(T)+\nu_{m,i}(T))^{2}]}.
\end{equation*}

Therefore, using the Cauchy-Schwarz inequality and Lemma \ref{lem_moment_bound}, there exists a nonnegative constant $Q_1(T)$ such that
\begin{equation}
\label{eq_mart_bound_final}
\E\left[\int_{S_1}^{S_2}\left(\sum_{j\neq i}\mu_{j\rightarrow i}\rho_{1,j}(s)+\nu_{m,i}(s)\right)\diff s\right]\leq \delta Q_1(T).
\end{equation}

Now, let us bound the rest of the right-hand side term in \eqref{eq_aldous_bound_1}.

We have, using Fubini's theorem for the first equality and the fact that the term under the integral is nonnegative for the upper bound,
\begin{equation*}
\begin{split}
&\E\left[\int_{S_1}^{S_2}\left( \int_0^t \sum_{n \neq m}\sum_{j \neq i}\mu_{j\rightarrow i}\one_{\{V_{(n,j)\rightarrow i}(s)=m\}} \lambda_{n,j}(s) \diff s +\left(r_i-\lambda_{m,i}(s)\right) \lambda_{m,i}(s) \diff s \right) \diff t\right] \\
&=\E\bigg[\int_{0}^{S_2}\left(S_2-(S_1\vee s)\right)\bigg( \sum_{n \neq m}\sum_{j \neq i}\mu_{j\rightarrow i}\one_{\{V_{(n,j)\rightarrow i}(s)=m\}} \lambda_{n,j}(s) \diff s\\
&+\left(r_i-\lambda_{m,i}(s)\right) \lambda_{m,i}(s) \bigg)\diff s\bigg]\\
&\leq \delta \int_{0}^{T}\E\left[\sum_{j \neq i}\mu_{j\rightarrow i}\lambda_{1,j}(s)+\lambda_{m,i}^{2}(s)-r_i\lambda_{m,i}(s)\right] \diff s.
\end{split}
\end{equation*}

Therefore, using Lemma \ref{lem_moment_bound}, we have the existence of a constant $Q_2(T)>0$ such that for all $(n,j)\in \{1,...,M\}\times\{1,...,K\}$,
\begin{equation*}
\E\left[\int_{S_1}^{S_2} c_{n,j}(s)\diff s\right] \leq \delta Q_2(T).
\end{equation*}
Similarly, we obtain the existence of a constant $Q_3(T)>0$ such that for all $(n,j)\in \{1,...,M\}\times\{1,...,K\}$,
\begin{equation*}
\E\left[\int_{S_1}^{S_2}d_{n,j}(s)\diff s\right] \leq \delta Q_3(T).
\end{equation*}
Combining the two previous bounds, we have the existence of a nonnegative constant $Q_4(T)$ such that
\begin{equation}
\label{eq_comp_bound_final}
\E\left[\int_{S_1}^{S_2}\left(\sum_{j \neq i}\mu_{j\rightarrow i}c_{1,j}(s)+d_{m,i}(s)\right)\diff s\right] \leq \delta Q_4(T).
\end{equation}

Finally, combining \eqref{eq_mart_bound_final} and \eqref{eq_comp_bound_final} in \eqref{eq_aldous_bound_1}, we obtain
\begin{equation*}
\P\left(|\lambda_{m,i}(S_2)-\lambda_{m,i}(S_1)|>\epsilon'\right) \leq \frac{\delta (Q_1(T)+Q_4(T))}{\epsilon'}.
\end{equation*}

Therefore, we can choose $\delta$ so that 
\begin{equation*}
\limsup_{N\rightarrow +\infty} \sup_{\substack{S_1,S_2 (\mathcal{F}_t^{(n)})\text{ s.t.} \\ S_1 \leq S_2 \leq (S_1+\delta) \wedge T}} \P(|\lambda_{m,i}^{N}(S_1)-\lambda_{m,i}^{N}(S_2)| > \epsilon ') \leq \epsilon.
\end{equation*}
This proves the second condition of the Aldous criterion.

For the first condition, for $t \in [0,T],$ let
\begin{equation*}
\begin{split}
G(t)&=\sum_{n \neq m}\sum_{j \neq i}\mu_{j\rightarrow i} \int_0^t \one_{\{V_{(n,j)\rightarrow i}(s)=m\}}(N_{n,j}(\diff s)-\lambda_{n,j}(s)\diff s)\\
&+\int_0^t(\lambda_{m,i}(s)-r_i) (N_{m,i}(\diff s)-\lambda_{m,i}(s)\diff s)
\end{split}
\end{equation*}
and 
\begin{equation*}
H(t)=\sum_{n \neq m}\sum_{j \neq i}\mu_{j\rightarrow i} \int_0^t \one_{\{V_{(n,j)\rightarrow i}(s)=m\}}\lambda_{n,j}(s)\diff s +\int_0^t(\lambda_{m,i}(s)-r_i) \lambda_{m,i}(s)\diff s.
\end{equation*}
As $G$ is a martingale, by Doob's inequality and by Lemma \ref{lem_moment_bound}, there exists a constant $K_1(T)$ such that
\begin{equation*}
\E\left[\sup_{t\in [0,T]}\left|G(t)\right|\right]\leq 4\E\left[\left|G(T)\right|^2\right] \leq K_1(T).
\end{equation*}

Moreover,since all terms under the integrals in $H(t)$ are non-negative, $H(t)$ is non-decreasing in $t$, so by Lemma \ref{lem_moment_bound}, there exists a constant $K_2(T)$ such that
\begin{equation*}
\E\left[\sup_{t \in [0,T]}\left|H(t)\right|\right]\leq \E[H(T)] \leq K_2(T).
\end{equation*}
By the triangular inequality,
\begin{equation*}
\begin{split}
\E\left[\sup_{[0,T]}\left|\lambda_{m,i}(t)\right|\right] &\leq \E\left[\sup_{[0,T]}\left|G(t)\right|\right]+\E\left[\sup_{[0,T]}\left|H(t)\right|\right] \\ 
 &\leq K_1(T)+K_2(T).
\end{split}
\end{equation*}

Thus, for all $M\geq 2$, if $\kappa>0$, by the Markov inequality,
\begin{equation*}
\begin{split}
\P\left(||\lambda^M_{m,i}||_{\infty}>\kappa\right) &\leq \frac{1}{\kappa}\E\left[\|\lambda^M_{m,i}\|_{\infty}\right] \\
&\leq \frac{K_1(T)+K_2(T)}{\kappa},
\end{split}
\end{equation*}
so for all $\epsilon>0$, there exists $\kappa>0$ such that for all $M \geq 2$,
\begin{equation*}
\P(\|\lambda_{m,i}^M\|_{\infty}>\kappa)<\epsilon,
\end{equation*}
which proves the first condition of the Aldous criterion.

Thus, both conditions of the Aldous criterion are verified, and the set of processes \\
$(\lambda^M_{m,i})_{(m,i)\in \{1,...,M\}\times\{1,...,K\}}$ is tight in the space of càdlàg functions on $\R^+$. Combining it with statement 4 of Theorem \ref{th_Poisson_CV} yields the result.
\end{proof} 

\section*{Acknowledgments}
The author would like to thank François Baccelli for his guidance and suggestions. The author would like to thank Sergey Foss, Eva Löcherbach and Gianluca Torrisi for feedback and suggestions.

\section*{Funding}
The author was supported by the ERC NEMO grant (\# 788851) to INRIA Paris.
\newpage
\bibliography{biblio}

\begin{thebibliography}{10}

\bibitem{ZAN_2022}
Z.~Agathe-Nerine.
\newblock Multivariate hawkes processes on inhomogeneous random graphs.
\newblock {\em Stochastic Processes and their Applications}, 152:86--148, 2022.

\bibitem{Aldous78}
D.~Aldous.
\newblock {Stopping Times and Tightness}.
\newblock {\em The Annals of Probability}, 6(2):335 -- 340, 1978.

\bibitem{gast_allmeier_22}
S.~Allmeier and N.~Gast.
\newblock Mean field and refined mean field approximations for heterogeneous
  systems: It works!
\newblock {\em Association for Computing Machinery}, 50(1):103–104, jun 2022.

\bibitem{Amblard_2004}
F.~Amblard and G.~Deffuant.
\newblock The role of network topology on extremism propagation with the
  relative agreement opinion dynamics.
\newblock {\em Physica A: Statistical Mechanics and its Applications},
  343:725–738, 11 2004.

\bibitem{Davydov2022}
F.~Baccelli, M.~Davydov, and T.~Taillefumier.
\newblock Replica-mean-field limits of fragmentation-interaction-aggregation
  processes.
\newblock {\em Journal of Applied Probability}, 59:1--22, 01 2022.

\bibitem{baccelli:2002}
F.~Baccelli, D.~R. Mcdonald, and J.~Reynier.
\newblock {A Mean-Field Model for Multiple TCP Connections through a Buffer
  Implementing RED}.
\newblock {\em Performance Evaluation}, pages 77--97, 2002.

\bibitem{Baccelli_2019}
F.~Baccelli and T.~Taillefumier.
\newblock Replica-mean-field limits for intensity-based neural networks.
\newblock {\em SIAM Journal on Applied Dynamical Systems}, 18(4):1756–1797, 1
  2019.

\bibitem{BarHolJan}
A.~Barbour, L.~Holst, and S.~Janson.
\newblock {\em Poisson approximation}.
\newblock Oxford Studies in Probability. Oxford University Press, 1992.

\bibitem{billingsley1968}
P.~Billingsley.
\newblock {\em Convergence of Probability Measures}.
\newblock Wiley Series in Probability and Mathematical Statistics. Wiley, 1968.

\bibitem{Bartek_notes}
B.~Blaszczyszyn.
\newblock {Lecture Notes on Random Geometric Models --- Random Graphs, Point
  Processes and Stochastic Geometry}.
\newblock Lecture, 12 2017.

\bibitem{BogKol2012}
V.~I. Bogachev and A.~V. Kolesnikov.
\newblock {The Monge-Kantorovich problem: achievements, connections, and
  perspectives}.
\newblock {\em Russian Mathematical Surveys}, 67(5):785--890, 10 2012.

\bibitem{Bremaud20}
P.~Br{\'e}maud.
\newblock {\em Point Process Calculus in Time and Space: An Introduction with
  Applications}.
\newblock Probability Theory and Stochastic Modelling. Springer International
  Publishing, 2020.

\bibitem{BremMass96}
P.~Br{\'e}maud and L.~Massouli{\'e}.
\newblock Stability of nonlinear hawkes processes.
\newblock {\em Annals of Probability}, 24:1563--1588, 1996.

\bibitem{Chen75}
L.~H.~Y. Chen.
\newblock {Poisson Approximation for Dependent Trials}.
\newblock {\em The Annals of Probability}, 3(3):534 -- 545, 1975.

\bibitem{Dobrushin1979}
R.~L. Dobrushin.
\newblock Vlasov equations.
\newblock {\em Functional Analysis and Its Applications}, 13(2):115--123, Apr
  1979.

\bibitem{Erny_21}
X.~Erny, E.~L{\"o}cherbach, and D.~Loukianova.
\newblock {Conditional propagation of chaos for mean field systems of
  interacting neurons}.
\newblock {\em Electronic Journal of Probability}, 26:1 -- 25, 2021.

\bibitem{Fournier_Locherbach_2016}
N.~Fournier and E.~L{\"o}cherbach.
\newblock {On a toy model of interacting neurons}.
\newblock {\em Annales de l'Institut Henri Poincaré, Probabilités et
  Statistiques}, 52(4):1844 -- 1876, 2016.

\bibitem{Galves_2013}
A.~Galves and E.~Löcherbach.
\newblock Infinite systems of interacting chains with memory of variable
  length—a stochastic model for biological neural nets.
\newblock {\em Journal of Statistical Physics}, 151(5):896–921, 3 2013.

\bibitem{Gast_2018}
N.~Gast, D.~Latella, and M.~Massink.
\newblock A refined mean field approximation of synchronous discrete-time
  population models.
\newblock {\em Performance Evaluation}, 126:1–21, 10 2018.

\bibitem{Kallenberg_2017}
O.~Kallenberg.
\newblock {\em Random Measures, Theory and Applications}, volume~77.
\newblock Springer, 01 2017.

\bibitem{Kanto42}
L.~Kantorovich.
\newblock On the translocation of masses.
\newblock {\em Journal of Mathematical Sciences}, 133, 03 2006.

\bibitem{Kleinrock_1}
L.~Kleinrock.
\newblock {\em {Queueing Systems}}, volume I: Theory.
\newblock Wiley Interscience, 1975.

\bibitem{Ramanan_2020}
D.~Lacker, K.~Ramanan, and R.~Wu.
\newblock Local weak convergence for sparse networks of interacting processes.
\newblock {\em Annals of Applied Probability, to appear}, 2020.

\bibitem{LeBoudec07}
J.-Y. {Le Boudec}, D.~McDonald, and J.~Mundinger.
\newblock A generic mean field convergence result for systems of interacting
  objects.
\newblock In {\em Proceedings of the Fourth International Conference on
  Quantitative Evaluation of Systems}, QEST ’07, page 3–18, USA, 2007. IEEE
  Computer Society.

\bibitem{LeCam86}
L.~Le~Cam.
\newblock {\em Asymptotic Methods in Statistical Decision Theory}.
\newblock Springer Series in Statistics. Springer, 1986.

\bibitem{lindvall2002}
T.~Lindvall.
\newblock {\em Lectures on the Coupling Method}.
\newblock Dover Books on Mathematics Series. Dover Publications, Incorporated,
  2002.

\bibitem{McKean_66}
H.~P. McKean.
\newblock A class of markov processes associated with nonlinear parabolic
  equations.
\newblock {\em Proceedings of the National Academy of Sciences},
  56(6):1907--1911, 1966.

\bibitem{Pastor_Satorras_2015}
R.~Pastor-Satorras, C.~Castellano, P.~Van~Mieghem, and A.~Vespignani.
\newblock Epidemic processes in complex networks.
\newblock {\em Reviews of Modern Physics}, 87(3):925–979, 8 2015.

\bibitem{Pillow2008}
J.~W. Pillow, J.~Shlens, L.~Paninski, A.~Sher, A.~M. Litke, E.~J. Chichilnisky,
  and E.~P. Simoncelli.
\newblock Spatio-temporal correlations and visual signalling in a complete
  neuronal population.
\newblock {\em Nature}, 454(7207):995--999, Aug 2008.

\bibitem{Robert2016}
P.~Robert and J.~Touboul.
\newblock On the dynamics of random neuronal networks.
\newblock {\em Journal of Statistical Physics}, 165(3):545--584, Nov 2016.

\bibitem{Shriki_2013}
O.~Shriki, J.~Alstott, F.~Carver, T.~Holroyd, R.~N. Henson, M.~L. Smith,
  R.~Coppola, E.~Bullmore, and D.~Plenz.
\newblock Neuronal avalanches in the resting meg of the human brain.
\newblock {\em Journal of Neuroscience}, 33(16):7079--7090, 2013.

\bibitem{Szn:89}
A.-S. Sznitman.
\newblock Topics in propagation of chaos.
\newblock {\em Ecole d'Ete de Probabilites de Saint-Flour XIX, vol.1464}, pages
  165--251, 1989.

\bibitem{teerapabolarn2014}
K.~Teerapabolarn.
\newblock Poisson approximation for random sums of poisson random variables.
\newblock {\em IJPAM}, 95(4):543--546, 2014.

\bibitem{truccolo_2005}
W.~Truccolo, U.~T. Eden, M.~R. Fellows, J.~P. Donoghue, and E.~N. Brown.
\newblock A point process framework for relating neural spiking activity to
  spiking history, neural ensemble, and extrinsic covariate effects.
\newblock {\em Journal of Neurophysiology}, 93:1074–1089, 2005.

\bibitem{Vladimirov_2018}
A.~A. Vladimirov, S.~A. Pirogov, A.~N. Rybko, and S.~B. Shlosman.
\newblock Propagation of chaos and {P}oisson hypothesis.
\newblock {\em Problems of Information Transmission}, 54(3):290–299, 7 2018.

\bibitem{VveDobKar96}
N.~Vvedenskaya, R.~Dobrushin, and F.~Karpelevich.
\newblock Queueing system with selection of the shortest of two queues: An
  asymptotic approach.
\newblock {\em Probl. Peredachi Inf.}, 32:20--34, 1996.

\end{thebibliography}
\bibliographystyle{abbrv}

\end{document}